\documentclass{article}
\usepackage{amssymb, amsthm, amsmath, amsfonts}
\usepackage{geometry}
\usepackage[colorlinks=true,urlcolor=blue,linkcolor=red,citecolor=magenta]{hyperref}
\usepackage{cleveref}
\usepackage{hyperref}

\usepackage{xcolor}

\newtheorem{theorem}{Theorem}[section]
\newtheorem{lemma}[theorem]{Lemma}
\newtheorem{corollary}[theorem]{Corollary}
\newtheorem{proposition}[theorem]{Proposition}
\newtheorem{conjecture}[theorem]{Conjecture}

\theoremstyle{definition}
\newtheorem{definition}[theorem]{Definition}
\newtheorem{remark}[theorem]{Remark}

\newtheorem{openp}[theorem]{Open Problem}

\DeclareMathOperator{\wt}{wt} 
\DeclareMathOperator{\im}{Im}
\DeclareMathOperator{\mult}{mult}
\DeclareMathOperator{\supp}{supp}
\DeclareMathOperator{\Tr}{Tr}

\newcommand{\graph}[1]{\mathcal{G}_{#1}}

\newcommand{\Z}{\mathbb{Z}}

\newcommand{\F}{\mathbb{F}}
\newcommand{\set}[1]{\left\{ #1 \right\}}

\newcommand{\parens}[1]{ \left( #1 \right)}

\newcommand{\floor}[1]{ \left\lfloor #1 \right\rfloor }
\newcommand{\bentcomps}[1]{{\mathcal{B}({#1})}}
\newcommand{\semibentcomps}[1]{{\mathcal{SB}({#1})}}
\newcommand{\nearbentcomps}[1]{{\mathcal{NB}({#1})}}

\begin{document}

\title{Resolving a conjecture on quadratic APN functions and a new quadratic $(n,n)$-function associated to crooked functions}

\author{Claude Carlet$^1$, Darrion Thornburgh$^2$}
\date{$^1$Universities of Paris 8, France and Bergen, Norway. \\ \texttt{claude.carlet@gmail.com} \\ 
$^2$Department of Mathematics, Vanderbilt University, USA.
\texttt{darrion.thornburgh@vanderbilt.edu}\\[2ex]
}
\maketitle

\begin{abstract}
We say an $(n,n)$-function $F \colon \mathbb{F}_2^n \to \mathbb{F}_2^n$ is a crooked function if for any nonzero $a \in \mathbb{F}_2^n$, the image of $D_aF(x)=F(x)+F(x+a)$ is an affine hyperplane.
The only known examples of crooked functions are all quadratic almost perfect nonlinear (APN), or equivalently, for every known crooked function, $D_aF$ is affine for all $a \in \mathbb{F}_2^n$.
The ortho-derivative $\pi_F \colon\mathbb{F}_2^n \to \mathbb{F}_2^n$ of a crooked function $F$ is the function such that $\pi_F(0)=0$, and for any nonzero $a$, the set $\{0,\pi_F(a)\}^\perp$ is the underlying vector space of $\mathrm{Im}(D_aF)$.
We prove that for $n \geq 4$ and a crooked function $F$, if $k$ is a non-negative integer such that $F$ has $2^k-1$ quadratic component functions, $\pi_F$ has at least $2^n-2^{n-k}$ component functions of algebraic degree $n-2$.
In particular, we resolve Gorodilova's conjecture that every component function of $\pi_F$ has algebraic degree $n-2$ when $F$ is quadratic APN.
As a second main result, for $n \geq 4$, we associate to a crooked function $F$ a quadratic function $\varepsilon_F \colon \mathbb{F}_2^n \to \mathbb{F}_2^n$ that satisfies a strong geometric-combinatorial condition regarding the sums of $F$ over $2$-dimensional linear subspaces.
As a corollary to both of our main results, we prove that for any even $n \geq 4$, any quadratic APN $(n,n)$-function has at least $n$ semi-bent components.
Furthermore, we obtain a congruence result on a problem on $m$-sequences introduced by Johansen, Helleseth, and Kholosha, and we determine the exact algebraic degrees of some Boolean functions associated to the bent and near-bent components of particular classes of plateaued vectorial functions.
\end{abstract}

\section{Introduction}\label{sec:introduction}

We call a function $F\colon\F_2^n \to \F_2^m$ an \textit{$(n,m)$-function}, and for any $v \in\F_2^m$, we say the function $(v \cdot F)(x) = v \cdot F(x)$ is a \textit{component function} of $F$.
For a point $a \in \F_2^n$ and an  $(n,m)$-function $F$, the function $D_a F(x)=F(x)+F(x+a)$ is the \textit{derivative of $F$ in the direction of $a$}, and the image of $D_a F$ is a \textit{differential set} of $F$.
An $(n,n)$-function $F$ is \textit{crooked} if the image of $D_a F$ is an affine hyperplane (that is, a translate $u+H$, with $u \in \F_2^n$ of a linear hyperplane $H \subseteq \F_2^n$) for all nonzero $a \in \F_2^n$ \cite{Kyureghyan2007}.
Crooked functions are known to have many interesting properties, including some graph-theoretical connections to distance-regular graphs \cite{BendingFlaassCrooked1998}.
However, all known crooked functions are quadratic ($F$ being quadratic is equivalent to $D_a F$ being affine for all $a\in \F_2^n$).
The central problem surrounding crooked functions is to find an example of a non-quadratic crooked function or prove no such function exists.

Crooked functions lie inside of a larger class of functions that are called \textit{almost perfect nonlinear} (APN). 
An $(n,n)$-function $F$ is said to be APN if the image set of $D_a F$ has size $2^{n-1}$ for all nonzero $a \in \F_2^n$, or equivalently, $D_a F$ is a $2$-to-$1$ function for all nonzero $a \in \F_2^n$.
Crooked functions are distinguished from non-crooked APN functions by the existence of an ortho-derivative.
For a crooked function $F$, the \textit{ortho-derivative} of $F$ is the unique function $\pi_F \colon \F_2^n \to \F_2^n$ such that $\pi_F(0)=0$ and $\set{0,\pi_F(a)}^\perp$ is the underlying vector space of $\im(D_aF)$ for all nonzero $a \in \F_2^n$.
Studying the ortho-derivative of a crooked function (in particular, the ortho-derivative of a quadratic APN function) has been of great interest in recent research (see, e.g. \cite{MihailaThornburgh2026,CouvreurOrtho-Derivative,BeierleLeanderExtensions,Gorodilova2019,CarletPiccioneStrongDProperty}).
For instance, the Walsh spectrum of the ortho-derivative of a quadratic APN function is regarded as strongly discriminating with respect to Carlet-Charpin-Zinoviev (CCZ) equivalence (see \Cref{prelim} for a complete background of definitions), and we refer the reader to 
\cite{beierle2025millionsinequivalentquadraticapn,BeierleLeander2022,Yu2022} for more practical utility of the ortho-derivative.
There remain several open problems on the ortho-derivatives of crooked functions (for instance, it is unknown whether or not two extended affine (EA) inequivalent crooked functions can have linearly equivalent ortho-derivatives).
In this paper, we are concerned with the following conjecture on the algebraic degrees of the components of the ortho-derivative of a quadratic APN function.

\begin{conjecture}[\cite{Gor20}]\label{conj:Gorodilova}
Assume $n \geq 4$, and let $F \colon \F_2^n\to \F_2^n$ be a quadratic APN function.
For any nonzero $v \in \F_2^n$, the algebraic degree of the component function $(v \cdot \pi_F)(x) = v \cdot \pi_F(x)$ is equal to $n-2$.
\end{conjecture}

To provide a brief history on \Cref{conj:Gorodilova}, Gorodilova first proved in \cite{Charpin2019DiffUni} that $\deg_{alg}(\pi_F)=\max_{v \in \F_2^n \setminus \set{0}} \deg_{alg}(v \cdot \pi_F) \leq n-2$ when $n \geq 3$ is odd.
Then, in \cite{Gor20}, Gorodilova introduced \Cref{conj:Gorodilova} and computationally observed that it held for all known (at the time) quadratic APN functions for all $n \leq 11$.
A few years later, in \cite{CouvreurOrtho-Derivative}, Couvreur, Canteaut, and Perrin proved $\deg_{alg}(\pi_F) \leq n-2$ for all $n \geq 3$ when $F$ is a quadratic APN function.
The problem of whether or not it was possible for the inequality $\deg_{alg}(\pi_F) <n-2$ to hold also remained open.

The first main result of this paper is that we prove \Cref{conj:Gorodilova} in the affirmative by proving that if $n \geq 4$ and $F$ is a crooked $(n,n)$-function with at least $2^k-1$ quadratic component functions\footnote{The set of those $v$ such that $v \cdot F$ is quadratic being a vector space, its size is a power of $2$.}, then $\pi_F$ has at least $2^n-2^{n-k}$ components of algebraic degree $n-2$.
First, we consider, more generally, a function $F$ whose differential sets are affine spaces (such a function is said to have \textit{the crooked property} \cite{Charpin2022}), and we study the indicators of the underlying vector spaces of the differential sets of $F$.
In particular, we define Boolean functions that are associated to a function $F$ with the crooked property, and these functions naturally generalize a key property of the component functions of the ortho-derivative of a crooked function.
Furthermore, we describe the Walsh transforms of the Boolean functions that we associate to functions with the crooked property.
From this, we obtain a description of the Walsh transform of the ortho-derivative of a crooked function $F$ in terms of the 3-sums of the graphs of $F$ and its restrictions to affine hyperplanes.
In \cite[Theorem 1]{CouvreurOrtho-Derivative}, the upper bound $\deg_{alg}(\pi_F) \leq n-2$ was shown when $F$ is a quadratic APN function.
We then prove more generally, using different methods than \cite{CouvreurOrtho-Derivative}, that we have $\deg_{alg}(\pi_F) \leq n-2$ for any crooked function $F$.

On the other hand, when we prove the inequality $\deg_{alg}(v \cdot \pi_F) \geq n-2$ for $F$ quadratic APN, our arguments will be more involved.
A key ingredient in our proof is the collection of ``exclude parity functions'' associated to quadratic APN functions. 
For an APN function $F \colon \F_2^n \to \F_2^n$ and a nonzero $v \in \F_2^n$, the Boolean function $f_v \colon \F_2^n \to \F_2$ defined by 
\[
f_v(u) \equiv |\{\set{x,y} \subseteq \set{0,u}^\perp : \beta_F(x,y)=v\}| \pmod 2,
\]
is an \textit{exclude parity function of $F$} (we introduce this terminology as these functions characterize the parity of ``exclude multiplicities'' with respect to some particular Sidon sets, see \Cref{prelim}), where we let $\beta_F(x,y) = F(0)+F(x)+F(y)+F(x+y)$.
We will show that these Boolean functions are always quadratic (i.e. of algebraic degree at most $2$), and we prove that if $F$ is crooked, then for any nonzero $v \in \F_2^n$, the component function $v \cdot \pi_F$ has algebraic degree $n-2$ if and only if $f_v$ has algebraic degree equal to $2$.
We then confirm this latter condition under some particular hypotheses that are satisfied by crooked functions with at least one quadratic component function.
In fact, we prove all of the above, more generally, for $(n,m)$-functions that are differentially 2-uniform on a linear subspace satisfying some particular conditions that generalize crookedness.

As our second main result, we prove that for all $n \geq 4$ and any crooked $(n,n)$-function $F$, there exists an associated $(n,n)$-function $\varepsilon_F$ such that 
\[
v \cdot \varepsilon_F(u) = f_v(u) \oplus 1
\]
for all nonzero $u,v \in \F_2^n$, and we call this function the \textit{exclude parity adjoint} of $F$.
Equivalently, for $n \geq 4$, any crooked function $F$, and any nonzero $u \in \F_2^n$, the set of nonzero $v \in \F_2^n$ such that $|\{\set{x,y} \subseteq \set{0,u}^\perp : \beta_F(x,y)=v\}|$ is even is the complement of a linear hyperplane or is empty.
Furthermore, in the case that $F$ is quadratic, we show that \Cref{conj:Gorodilova} implies that $\varepsilon_F(u)=0$ if and only if $u=0$, and so the set of nonzero $v \in \F_2^n$ such that $|\{\set{x,y} \subseteq \set{0,u}^\perp : \beta_F(x,y)=v\}|$ is even is always the complement of a linear hyperplane when $F$ is a quadratic APN function.
Moreover, the exclude parity adjoint of a crooked function is always quadratic since its component functions are all quadratic.

Interestingly, it is possible for $\varepsilon_F$ to be APN.
In particular, we prove for $n \geq 4$, if $F$ is a crooked $(n,n)$-function, then $\varepsilon_F$ is APN if and only if $\pi_F$ is $(n-2)$th order sum-free, i.e. the sum of $\pi_F$ over any $(n-2)$-dimensional affine subspace is nonzero.
Also, in the case that $n \geq 4$ and $F$ is a Gold APN function, we have $F = \varepsilon_F$.
Meanwhile, it is also possible for $\varepsilon_F$ to not be APN while still having similar properties to $F$.
For example, we compute that the exclude parity adjoint of the Kim APN function $\kappa$ in $6$ variables is a differentially 4-uniform function whose image is linearly equivalent to $\im(\kappa)$ (implying it is a partial difference set).

As a corollary to both of our main results, we prove that if $n \geq 4$ is even, and $F$ is a quadratic APN $(n,n)$-function, then $F$ has at least $n$ semi-bent components. 
We prove as a corollary to the fact that $\varepsilon_F(u)$ is equal to the sum of points $b \in \F_2^n$ such that $b \cdot F$ is semi-bent and $u$ is not contained orthogonal of the linear kernel of $b \cdot F$.

We also describe a known open problem on the exponential sums 
\[
G^{(i)}_n = \sum_{x \in \F_{2^n}^\ast} (-1)^{\Tr(x^{2^i+1} + x^{-1})}
\]
of \cite{JohansenHellesethKholoshamseq} through a combinatorial condition.
In particular, it is conjectured that $G^{(i)}_n = G^{(\gcd(i,n))}_n$, and we will prove (using our results on the exclude parity function of a Gold APN function) the congruence $G^{(i)}_n \equiv G^{(1)}_n \pmod{16}$ when $\gcd(i,n)=1$.

We then prove that if $n \geq 4$ is even and $F$ is a plateaued APN $(n,n)$-function, then $\deg_{alg}(1_{\bentcomps{F}}) = \frac{n}{2}$, where $\bentcomps{F|_H} = \set{v \in \F_2^n : v \cdot F|_H \text{ is bent}}$.
Moreover, we prove that if $n \geq 4$ is even, $F$ is a quadratic APN function, and $H$ is a linear hyperplane, then the existence and defining properties of $\varepsilon_F$ imply that $\deg_{alg}(1_{\nearbentcomps{F|_H}}) = n-1$, where $\nearbentcomps{F|_H} = \set{v \in \F_2^n : v \cdot F|_H \text{ is near-bent}}$.

\section{Preliminaries}\label{prelim}

For an \textit{$n$-variable Boolean function} $f \colon \F_2^n \to \F_2$ and an affine subspace $A \subseteq \F_2^n$, we define the \textit{Walsh transform of $f$ on $A$} (or the \textit{Walsh transform of $f|_A$}) as 
$
W_{f|_A}(u) = \sum_{x \in A}(-1)^{f(x) + x\cdot u},
$
for all $u \in \F_2^n$,
where ``$\cdot$'' denotes an inner product (that is, a non-degenerate symmetric bilinear form) on $\F_2^n$.
In the case that we identify $\F_2^n$ with $\F_{2^n}$, then we define $a \cdot b = \Tr(ab)$, where $\Tr(x) = \sum_{i=0}^{n-1} x^{2^i} \in \F_2$ is the \textit{absolute trace} of $x$.
For a function $F \colon \F_2^n \to \F_2^m$, called an \textit{$(n,m)$-function} or a \textit{vectorial Boolean function}, the \textit{Walsh transform} of $F$ on $A$ is defined by
$
W_{F|_A}(u,v) = \sum_{x \in A} (-1)^{u \cdot x + v \cdot F(x)}
$
for all $(u,v) \in \F_2^n \times \F_2^m$.
The \textit{Walsh transform} of $f$ (resp. the \textit{Walsh transform} of $F$), denoted by $W_f$ (resp. $W_F$), is simply its Walsh transform on $\F_2^n$.

\textit{Parseval's relation} is the equality $\sum_{u \in \F_2^n}W_f^2(u)=2^{2n}$ and can be extended to $\sum_{u \in \F_2^n} W_{f|_A}^2(u) = 2^n |A|$ for any set $A \subseteq \F_2^n$.
The \textit{linearity} of $f$ is $\mathcal{L}(f) = \max_{u \in \F_2^n} |W_f(u)|$, and clearly $\mathcal{L}(f) \geq 2^{\frac{n}{2}}$ by Parseval's relation.
A Boolean function with linearity $2^{\frac{n}{2}}$ is called \textit{bent} or \textit{maximally nonlinear}, and clearly bent functions only exist for $n$ even.
If there exists a non-negative integer $\lambda$ such that $W_f(u) \in \set{0,\pm \lambda}$ for all $u \in \F_2^n$, then we say that $f$ is \textit{plateaued} with \textit{amplitude} $\lambda$.
All bent functions are plateaued since they satisfy $|W_f(u)| = 2^{\frac n 2}$ for all $u \in \F_2^n$.

The \textit{support} of an $n$-variable Boolean function $f$ is the set $\set{x \in \F_2^n : f(x) = 1}$, and the size of its support is called the \textit{Hamming weight} $\wt(f)$ of $f$.
We say that $f$ is \textit{balanced} if its Hamming weight is $2^{n-1}$.
Any $n$-variable Boolean function $f$ has a unique \textit{algebraic normal form} (ANF)
\[
f(x) = \bigoplus_{I \subseteq \set{1, \dots, n}} a_I \prod_{i \in I}x_i
\in \F_2[x_1, \dots, x_n]/(x_1^2 \oplus x_1, \dots, x_n^2 \oplus x_n),
\]
where $a_I \in \F_2$, and where we denote addition in $\F_2$ by $\bigoplus, \oplus$ while we reserve $\sum,+$ for addition in $\Z$ or $\F_2^n$.
The \textit{algebraic degree} $\deg_{alg}(f)$ of $f$ is equal to the largest degree of any (nonzero) monomial in the ANF of $f$ when $f$ is not identically zero, and the zero function is defined to have algebraic degree zero.
Boolean functions with algebraic degree at most $0,1,2$ are called constant, affine, and quadratic, respectively.
When $f$ is non-constant, the \textit{derivative of $f$ in the direction of $a$}, that is $D_af(x)=f(x)\oplus f(x+a)$ where $a \in \F_2^n$, has algebraic degree at most $\deg_{alg}(f)-1$ (with at least one derivative of $f$ having algebraic degree equal to $\deg_{alg}(f)-1$), see \cite{CarletBook}. 
Hence, quadratic Boolean functions are characterized by having all of their derivatives being affine.
We say that $f$ is \textit{partially-bent} if $D_a f$ is constant or balanced for all $a \in \F_2^n$ \cite{Carlet1993}.
Any bent function is partially-bent because bent functions are equivalently characterized by the property of $D_a f$ being balanced for all nonzero $a$ \cite{ROTHAUS1976300}.
Clearly, quadratic Boolean functions are also partially-bent, and it turns out that all partially-bent functions are plateaued. 
The \textit{linear kernel} of a Boolean function $f$ is the linear subspace 
$
\mathcal{E}_f = \set{a\in \F_2^n : D_a f \text{ is constant}}.
$

Two $(n,m)$-functions $F$ and $G$ are said to be \textit{affinely equivalent} if there exist affine isomorphisms $\varphi \colon \F_2^n \to \F_2^n$ and $\psi \colon \F_2^m \to \F_2^m$ such that $G = \psi \circ F \circ \varphi$.
For any affine invariant property $N$ (e.g., plateauedness, bentness, etc.) or parameter $P$ (e.g., algebraic degree, nonlinearity, etc.) for a Boolean function $f$, we define $N(f|_A)$ and $P(f|_A)$ to be equal to $N(f \circ \varphi)$ and $P(f \circ \varphi)$, respectively, where $\varphi \colon \F_2^k \to A$ is an affine isomorphism.

Let $C_k^n$ be the set consisting of $n$-variable Boolean functions $f$ such that $f|_A$ is plateaued for all $k$-dimensional affine subspaces $A$.
The authors proved in \cite{CarletThornburghRestrictions} that for $n\geq4$, we have the equivalence $f\in C_{n-1}^n$ if and only if $f$ is partially-bent, and moreover, for $3 \leq k \leq n-2$, the set $C_k^n$ equals the set of quadratic Boolean functions.

For an $(n,m)$-function $F$ and $v \in \F_2^m$, we say that the $n$-variable Boolean function $v \cdot F$ defined by $(v \cdot F)(x)=v\cdot F(x)$ for all $x \in \F_2^n$ is a \textit{component function} (or simply, a \textit{component}) of $F$.
An $(n,m)$-function is called \textit{bent}, \textit{quadratic}, or \textit{plateaued}, if its component functions are all bent, quadratic, or plateaued, respectively.
It is called \textit{strongly plateaued} if all of its component functions are partially-bent \cite{CarletPlateaued} (it is then plateaued and its second-order derivatives have an additional property).
We also define the \textit{algebraic degree} $\deg_{alg}(F)$ of $F$ as $\max_{v \in \F_2^m} \deg_{alg}(v \cdot F)$, and we define constant, affine, and quadratic vectorial Boolean functions in the same way as we do for Boolean functions.
The \textit{linearity} of $F$ is also defined by $\mathcal{L}(F) = \max_{v \in \F_2^m \setminus \set{0}} \mathcal{L}(v \cdot F)$.
By the result of \cite{CarletThornburghRestrictions} recalled above, for $n \geq 4$, an $(n,m)$-function $F$ is strongly plateaued if and only if it is plateaued on all affine hyperplanes of $\F_2^n$, and moreover, $F$ is quadratic if and only if for any (or for some) $3 \leq k \leq n-2$ the restriction of $F$ to any $k$-dimensional affine subspace is plateaued.

For a \textit{pseudo-Boolean function} $\varphi \colon \F_2^n \to \Z$, we define the \textit{Fourier-Hadamard transform} of $\varphi$ by 
$
\widehat{\varphi}(u) = \sum_{x \in \F_2^n}(-1)^{u \cdot x} \varphi(x),
$
for all $u \in \F_2^n$.
The \textit{convolutional product} of two pseudo-Boolean functions $\varphi, \psi \colon \F_2^n \to \Z$ is given by 
$
(\varphi \otimes \psi)(a) = \sum_{x \in \F_2^n}\varphi(x)\psi(x+a),
$
and it is clear that $(\varphi \otimes \psi)(a) = (\psi \otimes \varphi)(a)$.
The convolution theorem tells us that $\widehat{\varphi \otimes \psi}(u) = \widehat{\varphi}(u) \widehat{\psi}(u)$ (see e.g. \cite[Proposition 11]{CarletBook}).
It is clear that the Walsh transform of an $(n,m)$-function $F$ is equal to the Fourier-Hadamard transform of the indicator $1_{\graph F}$ of its graph 
\[
\graph{F} = \set{(x,F(x)) : x \in \F_2^n} \subseteq \F_2^n \times \F_2^m,
\]
where the indicator $1_{\graph F}$ is defined by 
$
1_{\graph F}(x,y) = \begin{cases}
    1 & y = F(x), \\
    0 & \text{otherwise}.
\end{cases}
$
That is, $W_F(u,v) = \widehat{1_{\graph{F}}}(u,v)$ for any $(u,v) \in \F_2^n \times \F_2^m$, and the same property holds for a restriction of $F$ to an affine subspace $A \subseteq \F_2^n$, i.e. $W_{F|_A}(u,v) = \widehat{1_{\graph{F|_A}}}(u,v)$.

For an $(n,m)$-function $F$, we say that $F$ is \textit{differentially $\delta$-uniform} for a positive integer $\delta$ if, for any $a \in \F_2^n \setminus \set{0}$, the derivative $D_a F(x) = F(x)+F(x+a)$ maps to any single element in $\F_2^m$ at most $\delta$ times.
In particular, if $F$ is an $(n,n)$-function, then $F$ is APN if and only if $F$ is differentially 2-uniform.
APN functions are mathematical objects that have been under scrutiny since the 1990s, and they remain to be an active area of research with many difficult open problems (e.g., \textit{the big APN problem} is to determine the existence of an APN permutation over an even dimension in general, with only one known example for $n=6$).
Recall that an $(n,n)$-function is crooked if the image set $\im(D_aF)$ is an affine hyperplane for all $a\in \F_2^n$, and so all crooked functions are APN by definition.
Quadratic APN functions are examples of crooked functions (because their nonzero derivatives are all affine maps with image sets of size $2^{n-1}$), and it remains a large open problem whether non-quadratic crooked functions exist \cite{Kyureghyan2007}. 

Studying differentially 2-uniform $(n,m)$-functions is also of interest in additive combinatorics because they are those $(n,m)$-functions whose graph $\graph F$ is a Sidon set in $\F_2^n\times \F_2^m$.
A \textit{Sidon set} $S \subseteq \F_2^d$ is a subset that does not contain a 2-dimensional affine subspace of $\F_2^d$, or equivalently, no four distinct points in $S$ have trivial sum.
Determining the largest size of a Sidon set in $\F_2^d$ is known to be a very difficult problem \cite{CzerwinskiPottLargeSidon}, and it is analogous to the famous \textit{cap set problem} (that is, to determine the largest size of a set in $\F_3^d$ that does not contain an affine line).
We say that a Sidon set $S \subseteq \F_2^d$ is \textit{maximal} if it is not contained in some larger Sidon set.
For a Sidon set $S \subseteq \F_2^d$ and any $a\in \F_2^d \setminus S$, we write
\[
\mult_S(a) = |\set{\set{x,y,z} \subseteq S : x+y+z=a}|,
\]
and we say that $\mult_S(a)$ is the \textit{exclude multiplicity} of $a$.
Indeed, for $a \notin S$, the value of $\mult_S(a)$ is equal to the number of 2-dimensional affine subspaces in $S \cup \set{a}$, and so $\mult_S(a)>0$ for all $a \notin S$ if and only if $S$ is a maximal Sidon set.
It is also conjectured that if $S$ is equal to the graph of an APN function $F \colon \F_2^n \to \F_2^n$, then $\graph{F}$ is a maximal Sidon set if $n \geq 3$ \cite{Carlet_apnGraphMaximal}.

An interesting family of Sidon sets, for which we know they are maximal (when $n \geq 3$), consists of those Sidon sets that are equal to the graphs of plateaued APN functions \cite{budaghyanCarletHellesetUpperBoundsDegree,Carlet_apnGraphMaximal}.
Furthermore, if $F \colon \F_2^n \to \F_2^n$ is a plateaued APN function, then $\mult_{\graph F}(a,b) \geq 2^{\frac{n}{2}-1} - 1$ for all $(a,b) \notin \graph{F}$ \cite{MihailaThornburgh2026}.
We will also require the following proposition in the sequel.
\begin{proposition}[{\cite{MihailaThornburgh2026}}]\label{prop:plateaued-oddmults}
    Assume $n \geq 3$.
    Let $F \colon \F_2^n \to \F_2^n$ be a plateaued APN function.
    Then $\mult_{\graph F}(a,b)$ is odd for all $(a,b) \in (\F_2^n)^2 \setminus \graph{F}$.
\end{proposition}
A subclass of plateaued APN functions consists of those functions such that $W_F(u,v) \in \{0,\pm 2^{\frac{n+1}{2}}\}$ for all $(u,v) \in (\F_2^n)^2 \setminus \set{(0,0)}$, and we call such a function \textit{almost bent} (AB).

For any subsets $S,T \subseteq \F_2^d$ and $a \in \F_2^d$, it is clear that $(1_S \otimes 1_T)(a)$ equals
\[
\sum_{x \in \F_2^d} 1_S(x) 1_T(x+a) 
= |\set{x \in S : x+a \in T}| 
= |\set{(x_1, x_2) \in S \times T : x_1 + x_2 = a}|.
\]
More generally, by induction on $k$, for any subsets $S_1, \dots, S_k \subseteq \F_2^d$, we have for any point $a \in \F_2^d$ that
\begin{equation}\label{eq:set-conv-equality}
(1_{S_1} \otimes 1_{S_2} \otimes \cdots \otimes 1_{S_k})(a) = 
|\set{(x_1, \dots, x_k) \in S_1 \times \cdots \times S_k : x_1 + \cdots + x_k=a}|.
\end{equation}
From \cref{eq:set-conv-equality}, we see that for a Sidon set $S \subseteq\F_2^d$ and $a \in \F_2^d \setminus S$, we have 
\begin{equation}\label{eq:multS-cubesFHT}
\mult_S(a) = \frac{1}{6} (1_S \otimes 1_S \otimes 1_S)(a) = \frac{1}{6 \cdot 2^d} \widehat{(\widehat{1_S})^3}(a),
\end{equation}
with the second equality following from the convolution theorem and the fact that $\widehat{\widehat{\varphi}}=2^d\varphi$ for any pseudo-Boolean function $\varphi\colon \F_2^d \to \Z$.

It is also straightforward that $(1_S\otimes1_S \otimes 1_S)(a)\geq 3|S|-2$ for any $a \in S$, with equality holding for all $a \in S$ if and only if $S$ is Sidon.
Indeed, for any $a \in S$, the value of $(1_S\otimes1_S \otimes 1_S)(a)$ is $|\set{(x,y,z) \in S^3: x+y+z=a}|$.
This set includes the triple $(a,a,a)$ or some permutation of $(a,b,b)$, where $b \in S \setminus \set{a}$, and it contains only these triples when $S$ is Sidon.

\begin{remark}\label{rem:plateaued-restriction-mults}
Let us consider an $(n,m)$-function $F$ that is differentially 2-uniform and plateaued on an affine subspace $A \subseteq \F_2^n$, and let $S = \graph F \cap (A \times \F_2^m) = \graph{F|_A}$.
For any $v \in \F_2^m$, denote by $\lambda_{v,A}$ the amplitude of $v \cdot F|_A$.
Since $W_{F|_A}^3(u,v) = \lambda_{v,A}^2 W_{F|_A}(u,v)$ for any $(u,v) \in \F_2^n \times \F_2^m$, we have for any $t \in \F_2^n$ and $c\in \F_2^m \setminus \set{0}$ that
\begin{align*}
\mult_S(t,F(t)+c) 
&= \frac{1}{6 \cdot 2^{n+m}}\widehat{W_{F|_A}^3}(t,F(t)+c)\\
&=\frac{1}{6 \cdot 2^{n+m}} \sum_{(u,v) \in \F_2^n \times \F_2^m} (-1)^{u \cdot t + v \cdot (F(t)+c)} W^3_{F|_A}(u,v) \\
&= \frac{1}{6 \cdot 2^{n+m}}\sum_{v \in \F_2^m} (-1)^{v \cdot (F(t)+c)}\lambda_{v,A}^2 \sum_{x \in A}(-1)^{v \cdot F(x)} \sum_{u \in \F_2^n}(-1)^{u \cdot (t+x)}\\
    &= \frac{1_A(t)}{6 \cdot 2^m} \sum_{v \in \F_2^m} (-1)^{v \cdot c} \lambda_{v,A}^2,
\end{align*}
with the last equality following from the fact that $\sum_{u \in \F_2^n} (-1)^{u \cdot x}=2^n \cdot 1_{\set{0}}(x)$.
This implies that for any nonzero $c \in \F_2^m$, the value of $\mult_S(t,F(t)+c)$ is constant as $t$ ranges across $A$.
It was shown in \cite[Theorem 2]{CarletThornburghRestrictions} (and discussed in Remark 5 of this reference) that if an $n$-variable plateaued Boolean function $f$ is plateaued on a hyperplane $H \subseteq \F_2^n$, then $f|_{H^c}$ is plateaued as well, and both $f|_H$ and $f|_{H^c}$ have the same amplitude.
Assuming that $F$ is plateaued and $H$ is an affine hyperplane, we then know that $\lambda_{v,H}= \lambda_{v,H^c}$ for all $v \in \F_2^m$.
Hence, taking $S' = \graph{F} \cap (H^c \times \F_2^m) = \graph{F|_{H^c}}$, we have $\mult_S(t,F(t)+c) = \mult_{S'}(t',F(t')+c)$ for all $c \in \F_2^m \setminus \set{0}$, $t \in H$, and $t' \in H^c$.
\end{remark}

\begin{remark}\label{rem:amplitude-sums-Sidon-iff}
Let $F$ be an $(n,m)$-function that is plateaued on an affine subspace $A \subseteq \F_2^n$, and for any $v \in \F_2^m$, denote by $\lambda_{v,A}$ the amplitude of $v \cdot F$ on $A$.
We can prove the equivalence that
 \[
    \sum_{v \in \F_2^m} \lambda_{v, A}^2 \geq  2^m(3|A|-2)
    \]
    with equality if and only if $S = \graph F \cap (A \times \F_2^m)$ is a Sidon set.
    Indeed, denoting by $\supp(W_{v \cdot F|_A})$ the set $\set{u \in \F_2^n : W_{v \cdot F|_A}(u) \neq 0}$, we have upon applying Parseval's relation that
    \[
    \sum_{(u,v) \in \F_2^n \times \F_2^m} W_{F|_A}^4(u,v) = \sum_{v \in \F_2^m} \lambda_{v,A}^4 |\supp(W_{v \cdot F|_A})| =2^n|A| \sum_{v \in \F_2^m} \lambda_{v,A}^2.
    \]
    By \cite[Proposition 2.1]{CarletMesnager2022}, we know that $\sum_{(u,v) \in \F_2^n \times \F_2^m} (\widehat{1_S})^4(0,0) \geq 2^{n+m}(3|A|^2 - 2|A|)$ if and only if $S$ is Sidon.
    Since $\widehat{1_S}(u,v) = W_{F|_A}(u,v)$, we then have that $\sum_{v \in \F_2^m}\lambda_{v,A}^2 \geq 2^m(3|A|-2)$ with equality if and only if $S$ is Sidon.
\end{remark}

\section{Boolean functions associated to functions with the crooked property}\label{sec:associated-fcns}

For an $(n,m)$-function $F$, we say that $F$ has the \textit{crooked property} if all of its differential sets are affine spaces \cite{Charpin2022}; that $F$ is \textit{strongly plateaued} if all of its component functions are partially-bent \cite{CarletPlateaued}; that $F$ is \textit{crooked of codimension $k$} if all nonzero differential sets of $F$ are affine spaces of codimension $k$ \cite{StructuralWeaknessPermutations} (in the case that $n=m$, a function is crooked if and only if it is crooked of codimension $1$).
The three classes of functions satisfying these properties are related, and we require the following proposition (which strengthens a result of \cite{CarletPlateaued}) to see why.

\begin{proposition}\label{prop:strongplat-char} Any $(n,m)$-function $F$ is strongly plateaued if and only if the image set $\im(D_aF)=(D_aF)(\F_2^n)$ of any derivative $D_aF$ is an affine space, and $D_aF$ is balanced from $\F_2^n$ to this affine space.
\end{proposition}
\begin{proof}
As shown in \cite{CarletPlateaued}, $F$ is strongly plateaued if and only if for every $(a,w) \in \F_2^n \times \F_2^m$, the size of the set $\set{b \in \F_2^n : D_a F(b) = D_a F(x)+w}$ does not depend on the choice of $x \in \F_2^n$.
Let $F$ be strongly plateaued. 
Then, $D_a F$ matches at least once $D_aF(x)+w$ ({\em i.e.}, we have $w\in D_aF(x)+\im(D_aF)$) if and only if it matches  at least once $D_aF(y)+w$ ({\em i.e.}, we have $w\in D_aF(y)+\im(D_aF)$).  Hence, the set $\im(D_aF)$ is invariant under translation by any element of $\im(D_aF)+\im(D_aF)$ and is then  an affine space. Moreover, taking $w=0$, we have that for every $x,y$, the values $D_aF(x)$ and $D_aF(y)$ are matched the same number of times by $D_aF$. Conversely,  if for every $a$, the image set $\im(D_aF)$ is an affine space and every element in it is matched the same number of times by $D_aF$, then  $D_aF$ matches the same number of times any two values $D_aF(x)+w$ and $D_aF(y)+w$, since either $w$ belongs to the underlying space and $D_aF(x)+w$ and $D_aF(y)+w$ both belong to $\im(D_aF)$ or $w$ is outside the underlying space and the number of matches equals 0 in both cases. 
\end{proof}

Hence, strongly plateaued function has the crooked property \cite{CarletPlateaued}. 
By \Cref{prop:strongplat-char}, it is straightforward that if $F$ is a strongly plateaued $(n,m)$-function, then there exists a $k$ such that $F$ is crooked of codimension $k$ if and only if $|\set{\delta_F(a,b) : a \in \F_2^n \setminus \set{0}, b\in \F_2^m}|=2$, i.e. $F$ is \textit{differentially 2-valued} (for literature on differentially 2-valued functions, see  \cite{Tang2020,CHARPIN2019188,Charpin2019DiffUni,Peng2020,Blondeau2010,KolschPolujan2026}). 
Also, note that if $F$ is differentially 2-uniform, then $F$ has the crooked property if and only if it is strongly plateaued.

In general, the three classes containing functions with the crooked property, strongly plateaued functions, and functions that are crooked of codimension $k$ are closed under CCZ equivalence.
For example, the Kim APN function $x \mapsto x^3 + x^{10} + \alpha x^{24}$ defined on $\F_{2^6}$, where $\alpha \in \F_{2^6}^\ast$ is a primitive element, is a crooked function (because it is quadratic APN) that is CCZ equivalent to a permutation \cite{browningMcQuisttanWolfe}. 
This example is illustrative as an APN permutation of an even number of variables cannot have a nontrivial partially-bent component \cite{Calderini2017}.
However, the component functions of crooked functions are all partially-bent.

For an $(n,m)$-function $F$ with the crooked property, denote by $U_a$ the underlying vector space of $\im(D_aF)$ for all $a \in \F_2^n$.
In this section, we study the indicators of these linear subspaces, and we show that they can describe various properties of $F$.

Let us first provide a motivating example.
Assume that $F$ is a crooked $(n,n)$-function.
Recall that the ortho-derivative of $F$ is the uniquely defined function $\pi_F \colon \F_2^n \to \F_2^n$ such that $\pi_F(0)=0$ and $\set{0,\pi_F(a)}^\perp=U_a$ for all nonzero $a \in \F_2^n$.
For any $v \in \F_2^n$, the function $v \cdot \pi_F$ of takes value $0$ at $a \in \F_2^n \setminus \set{0}$ if and only if $v \in U_a$.
Note that if $a =0$, then $v \cdot \pi_F(a)=v\cdot \pi_F(0)=0$ for all $v \in \F_2^n$.
This provides the equality
\[
v\cdot \pi_F(a) = \begin{cases}
    0 & a = 0, \\
    1_{\F_2^n \setminus U_a}(v) & a \neq 0
\end{cases}
\]
which holds for all $a,v \in \F_2^n$.

We are interested in this particular relation satisfied by the component functions of the ortho-derivatives of crooked functions, and we wish to modify it in two directions: first we wish to consider the restriction of a given function $F$ to an affine subspace $A$ of the domain of $F$, and second we consider the case where the image sets of the derivatives of this restriction are affine subspaces of the codomain of $F$ (without any restriction on their dimensions).
For an $(n,m)$-function $F$ and an affine subspace $A \subseteq \F_2^n$ with underlying vector space $L$, we say that $F$ has the \textit{crooked property on $A$} if $\im (D_a F|_A)$ is an affine space for all $a \in L$.

\begin{definition}\label{def:generalized-ortho}
    Let $A \subseteq \F_2^n$ be an affine subspace with underlying vector space $L$. 
    Let $F$ be an $(n,m)$-function with the crooked property on $A$, and denote by $U_{a,A}$ the underlying vector space of $\im (D_a F|_A)$ for all $a \in L$.
    For any $v \in \F_2^m$, define the Boolean function
    \[
    \pi_{F|_A}^v \colon L \to \F_2
    \quad \text{such that} \quad 
    \pi_{F|_A}^v(a)=
    \begin{cases}
        0 & a = 0, \\
        1_{\F_2^m \setminus U_{a,A}}(v) & a \neq 0,
    \end{cases}
    \]
    for all $a \in L$.
\end{definition}

One may then ask when the Boolean functions $\pi_{F|_A}^v$ can equal the component function $v \cdot G$ of some vectorial function $G$, that we can denote by $\pi_{F|_A}$ (that is, when $\pi_{F|_A}^v + \pi_{F|_A}^{v'}$ always equals $\pi_{F|_A}^{v+v'}$.
Then, we have $\pi_{F|_A}^v = v \cdot \pi_{F|_A}$.
In such a case, for any nonzero $a \in L$, the function $1_{\F_2^m \setminus U_{a,A}}$ is linear, implying $U_{a,A}$ is equal to $\F_2^m$ or is the complement of a linear hyperplane of $\F_2^m$.
If $a \in L$ is such that $\im(D_aF|_A)=\F_2^m$, then we know $\dim(L)\geq m+1$.
Otherwise, if $a \in L$ such that $\im(D_aF|_A)$ is an affine hyperplane of $\F_2^m$, then we know $\dim(L) \geq m$.
So, if $n=m$ and $F$ is crooked, then $\dim(L)=n$, so $L = \F_2^n$ and $\pi_{F|_A}= \pi_F$.

\begin{remark}\label{rem:cpts-ortho-generalized}
    Let $F \colon \F_{2^n} \to \F_{2^n}$ be defined by $F(x)=x^3$, let $v \in \F_{2^n}^\ast$, and let $L\subseteq \F_{2^n}$ be a $k$-codimensional linear subspace.
    Throughout this example, we always consider $a$ to be in $L 
    \setminus \set{0}$.
    Since $D_a F(x)$ equals $a^2x+ax^2$ plus a constant, we have $U_{a,L} = \set{a^2x+ax^2 : x \in L}$.
    Note $U_{a,L} = a^3 \set{y^2 + y : y \in a^{-1} L}$.
    Let $H = \set{y \in \F_{2^n} : \Tr(y) = 0}$ (it is well-known that $H$ is the image of the $2$-to-$1$ linear map $x \mapsto x^2 + x$), and let $S \colon H \to \F_{2^n}$ be an affine map such that $S(b)^2 + S(b)=b$ for all $b \in H$ (e.g., after fixing an element $d$ such that $\Tr(d)=1$, that is $d \notin H$, we can take $S(b) = \sum_{j=1}^{n-1} b^{2^j} \sum_{\ell=0}^{j-1} d^{2^\ell}$ or $S(b) =1+ \sum_{j=1}^{n-1} b^{2^j} \sum_{\ell=0}^{j-1} d^{2^\ell}$, see \cite[Section 14.5]{CarletBook}).
    We have $z \in U_{a,L}$ if and only if $y^2 + y = \frac{z}{a^3}$ for some $y \in a^{-1} L$.
    Note that if $\frac{z}{a^3} \in H$, then the two solutions to $y^2 +y = \frac{z}{a^3}$ are $S\parens{\frac{z}{a^3}}$ and $1+S\parens{\frac{z}{a^3}}$.
    Therefore, $z \in U_{a,L}$ if and only if $\Tr\parens{\frac{z}{a^3}}=0$ and $S\parens{\frac{z}{a^3}} \in a^{-1}L$.
    Taking $z = v$, this gives $\pi_{F|_L}^v(a) =0$ if and only if $\Tr\parens{\frac{v}{a^3}}=0$ and $aS\parens{\frac{v}{a^3}} \in L$.
    Setting $L = \langle e_1, \dots, e_k\rangle^\perp$ for some linearly independent $e_1, \dots, e_k \in \F_{2^n}$, we then have 
    \[
    \pi_{F|_L}^v(a) 
    = 
    \begin{cases}
        1 & \Tr\parens{\frac{v}{a^3}}=1, \\
    1 \oplus \prod_{i=1}^k \parens{1 \oplus \Tr\parens{e_i a S\parens{\frac{v}{a^3}}}} & \Tr\parens{\frac{v}{a^3}}=0.
    \end{cases}
    \]
    By fixing $d \in \F_{2^n}$ with $\Tr(d)=1$ and taking $S(b) = \sum_{j=1}^{n-1} b^{2^j} \sum_{\ell=0}^{j-1} d^{2^\ell}$ for any $b \in \F_{2^n}$ (the value of $S(b)$ only matters when $b \in H$, but for notational purposes we extend $S$ to $\F_{2^n}$), we have  
    \[
    \pi_{F|_L}^v(a) = 1 
    \oplus \parens{1 \oplus \Tr\parens{\frac{v}{a^3}}}
    \prod_{i=1}^k \parens{1 \oplus \Tr\parens{e_i a 
    \sum_{j=1}^{n-1} \parens{\frac{v}{a^3}}^{2^j} \sum_{\ell=0}^{j-1} d^{2^\ell}
    }}
    \]
    for all $a \in L\setminus \set{0}$.
    When $L = \F_{2^n}$, the above equation exactly recovers the well-known fact that the ortho-derivative of $F(x) = x^3$ has component functions given by $\Tr(v\pi_F(a)) = \Tr\parens{\frac{v}{a^3}}$ \cite{CCZ} (where we define $0^{-1}:=0$).
\end{remark}

We now prove the following key lemma which will lead to us providing a combinatorial description of the Walsh spectrum of the ortho-derivative of a crooked function (we also note that this lemma generalizes \cite[Proposition 5.10]{CarletPiccioneStrongDProperty}).

\begin{lemma}\label{lem:ortho-cpts-Walsh1}
    Let $A \subseteq \F_2^n$ be an affine subspace, and let $F$ be an $(n,m)$-function that has the crooked property on $A$.
    If $v \in \F_2^m$ is nonzero, then  
    \begin{align*}
    W_{\pi_{F|_A}^v}(u) 
    &=2\sum_{a \in \Lambda_v} (-1)^{u \cdot a} + 2 - |L|1_{L^\perp}(u)\\
    &= 2\cdot \widehat{1_{\Lambda_v}}(u) + 2 - |L|1_{L^\perp}(u),
    \end{align*}
    for all $u \in \F_2^n$, where $\Lambda_v = \set{a \in L \setminus \set{0} : v \in U_{a,A}}$ and $U_{a,A}$ is the underlying vector space of $\im(D_aF|_A)$.
    \end{lemma}
\begin{proof}
Let $v \in \F_2^m$ be nonzero, let $L$ be the underlying vector space of $A$, and for all $a \in L$, let $U_{a,A}$ be the underlying vector space of $\im (D_a F|_A)$.
Applying the equality $(-1)^{\pi_{F|_A}^v} = 1-2\pi_{F|_A}^v$, for any $u\in \F_2^n$, the value of $W_{\pi_{F|_A}^v}(u)$ equals 
\[
\sum_{a \in L}(-1)^{\pi_{F|_A}^v(a) + u \cdot a}=
|L| 1_{L^\perp}(u) - 2 \sum_{a \in L}(-1)^{u \cdot a} \pi_{F|_A}^v(a).
\]
Since $\pi_{F|_A}^v(a)=1$ if and only if $a \in L \setminus \set{0}$ such that $v \notin U_{a,A}$, we have $W_{\pi_{F|_A}^v}(u) = |L| 1_{L^\perp}(u) - 2 \sum_{a \in L\setminus \set{0}, v \notin U_{a,A}}(-1)^{u \cdot a}$.
From the equality $\set{a \in L \setminus \set{0} : v \notin U_{a,A}} = (L\setminus \set{0}) \setminus \Lambda_v$, we have 
\begin{align*}
\sum_{a \in L\setminus \set{0}, v \notin U_{a,A}}(-1)^{u \cdot a}
&= \sum_{a \in L\setminus \set{0}}(-1)^{u \cdot a} -\sum_{a \in \Lambda_v} (-1)^{u \cdot a} \\
&= |L| 1_{L^\perp}(u) -1 -\sum_{a \in \Lambda_v} (-1)^{u \cdot a}.
\end{align*}
Hence, $W_{\pi_{F|_A}^v}(u) =2\sum_{a \in \Lambda_v} (-1)^{u \cdot a} + 2 - |L|1_{L^\perp}(u)$.
\end{proof}

We will see in the following lemma that \Cref{lem:ortho-cpts-Walsh1} can be used to provide a more refined description of the Walsh transform of $\pi_{F|_A}^v$ when $F$ has the crooked property and is differentially 2-valued on $A$ (equivalently, $F|_A$ is strongly plateaued and crooked of codimension $k$ for some $k$).
However, we have only defined differentially 2-valued $(n,m)$-functions and not differentially 2-valued restrictions of $(n,m)$-functions.
So, for an $(n,m)$-function $F$ and an affine subspace $A$ with underlying vector space $L$, we say that $F$ is \textit{differentially 2-valued on $A$} if the function $\delta_{F|_A} \colon \F_2^n \times \F_2^m \to \Z_{\geq0}$ defined by 
    \[
    \delta_{F|_A}(a,b) = \begin{cases}
        |\set{x \in A : F(x)+F(x+a)=b}| & a \in L, \\
        0 & a \notin L,
    \end{cases}
    \]
for all $(a,b) \in \F_2^n \times \F_2^m$, takes at most two distinct values on $(\F_2^n \setminus \set{0}) \times \F_2^m$.
Note that it is a straightforward consequence of \cref{eq:set-conv-equality} that $\delta_{F|_A} = 1_S \otimes 1_S$ where $S$ is the graph of $F|_A$.
We say that $F$ is \textit{differentially $\set{0,\Delta}$-valued on $A$} if for any nonzero $a \in \F_2^n$ (or equivalently any nonzero $a \in L$) and any $b \in \F_2^m$, we have $\delta_{F|_A}(a,b) \in \set{0,\Delta}$.
Since $\delta_{F|_A} = 1_S \otimes 1_S$ where $S = \graph{F|_A}$, it is clear that $F$ is differentially $\set{0,\Delta}$-valued on $A$ if and only if $|S \cap ((a,b) + S)| \in \set{0, \Delta}$ for all $(a,b) \in (\F_2^n \setminus \set{0}) \times \F_2^m$.

\begin{lemma}\label{lem:ortho-cpts-Walsh2}
    Let $A \subseteq\F_2^n$ be an affine subspace, and let $F$ be an $(n,m)$-function that has the crooked property on $A$ and is differentially $\set{0,\Delta}$-valued on $A$.
    Let $S = \graph{F|_A}$.
    For any $u \in \F_2^n$ and $t \in A$, let $K_{u,t} =(t+\set{0,u}^\perp) \times \F_2^m$.
    If $v \in \F_2^m$ is nonzero and $t \in A$, then 
    $W_{\pi_{F|_A}^v}(u)$ equals
    \begin{align*}
     \frac{4}{\Delta}(1_S \otimes 1_S \otimes 1_{S \cap K_{u,t}})(t,F(t)+v) -\frac{1}{\Delta \cdot 2^{n+m-1}} \widehat{(\widehat{1_S})^3}(t,F(t)+v)  + 2 - |A|1_{L^\perp}(u)
    \end{align*}
    for any $u \in \F_2^n$.
    In particular, $W_{\pi_{F|_A}^v}(0)=\frac{1}{\Delta \cdot 2^{n+m-1}} \widehat{(\widehat{1_S})^3}(t,F(t)+v) + 2 - |A|$.
\end{lemma}
\begin{proof}
For any nonzero $v \in \F_2^m$, let $\Lambda_v = \set{a \in L \setminus \set{0} : v \in U_{a,A}}$.
Note that $U_{a,A} = \im(D_aF|_A) + D_aF(t)$ for any $t \in A$.
Therefore, $a \in \Lambda_v$ if and only if $a \in L \setminus \set{0}$ and for any $t \in A$, there exists $b_t \in A \setminus \set{t}$ satisfying $v=F(t)+F(a+t)+F(b_t)+F(a+b_t)$.
That is, $a \in \Lambda_v$ if and only if $\delta_{F|_A}(a,F(t)+F(a+t)+v)>0$ for any $t \in A$.

Fix $t \in A$.
Then $2\sum_{a \in \Lambda_v} (-1)^{u \cdot a} = \frac{2}{\Delta} \sum_{a \in L}(-1)^{u \cdot a}\delta_{F|_A}(a, F(t)+F(t+a)+v)$ because $F$ is differentially $\set{0,\Delta}$-valued on $A$. 
Applying \Cref{lem:ortho-cpts-Walsh1}, we then know
\[
W_{\pi_{F|_A}^v}(u) = \frac{2}{\Delta} \sum_{a \in L}(-1)^{u \cdot a}\delta_{F|_A}(a, F(t)+F(t+a)+v) + 2 -|L|1_{L^\perp}(u)
\]
for all $u \in \F_2^n$.
Since $\delta_{F|_A} = 1_S \otimes 1_S$, we have 
\begin{align*}
\sum_{a \in L}  \delta_{F|_A}(a, F(t)+F(t+a)+v) 
&= \sum_{(x,F(x)) \in S} \delta_{F|_A}(x+t,F(t)+F(x)+v)\\
&= (1_S \otimes 1_S \otimes 1_S)(t,F(t)+v).
\end{align*}
By the convolution theorem, we have that $1_S \otimes 1_S \otimes 1_S = \frac{1}{2^{n+m}} \widehat{(\widehat{1_S})^3}$.
Also, we have that the sum $\sum_{a \in L}(-1)^{u \cdot a}\delta_{F|_A}(a, F(t)+F(t+a)+v)$ is equal to 
\begin{align*}
    &2\sum_{a \in L \cap \set{0,u}^\perp} \delta_{F|_A}(a, F(t)+F(t+a)+v) - \sum_{a \in L}  \delta_{F|_A}(a, F(t)+F(t+a)+v) \\
    &= 2 \sum_{a \in L \cap \set{0,u}^\perp} (1_S \otimes 1_S)(a, F(t)+F(t+a)+v)  -\frac{1}{2^{n+m}} \widehat{(\widehat{1_S})^3}(t,F(t)+v) \\
    &= 2 \sum_{x \in A \cap (t+\set{0,u}^\perp)} (1_S \otimes 1_S)(x+t, F(t)+F(x)+v)  - \frac{1}{2^{n+m}} \widehat{(\widehat{1_S})^3}(t,F(t)+v).
\end{align*}
By the definition of the convolutional product, this is equal to 
\[
2(1_S \otimes 1_S \otimes 1_{S \cap K_{u,t}})(t,F(t)+v) -\frac{1}{2^{n+m}} \widehat{(\widehat{1_S})^3}(t,F(t)+v),
\]
and the first claim follows.
For the second claim, by considering the case that $u=0$, we have $K_{u,t}=K_{0,t}=\F_2^n \times \F_2^m$, so
\begin{align*}
W_{\pi_{F|_A}^v}(0)
&= \frac{4}{\Delta \cdot 2^{n+m}} \widehat{(\widehat{1_S})^3}(t,F(t)+v) -
\frac{1}{\Delta \cdot 2^{n+m-1}} \widehat{(\widehat{1_S})^3}(t,F(t)+v)  + 2 - |A| \\
&= \frac{1}{\Delta \cdot 2^{n+m-1}} \widehat{(\widehat{1_S})^3}(t,F(t)+v) + 2 - |A|.
\end{align*}
\end{proof}

Under the same notation and assumptions as \Cref{lem:ortho-cpts-Walsh2}, in the case that $S = \graph{F|_A}$ is Sidon, we have that $\Delta = 2$ and $\widehat{(\widehat{1_S})^3}(t,F(t)+v) = 6 \cdot 2^{n+m} \mult_S(t,F(t)+v)$ for any $t \in A$ and nonzero $v \in \F_2^m$. 
In particular, if $F \colon \F_2^n \to \F_2^m$ has the crooked property on an affine subspace $A$ and is differentially $\set{0,2}$-valued on $A$, then it follows that 
\begin{equation}\label{eq:ortho-weight-mult}
    \mult_S(t,F(t)+v) = \frac{|A|-1-\wt(\pi_{F|_A}^v)}{3}
\end{equation}
for any $t \in A$ and $v \in \F_2^m \setminus \set{0}$ because $W_{\pi_{F|_A}^v}(0) = |A| - 2\wt(\pi_{F|_A}^v)$.
Relation~(\ref{eq:ortho-weight-mult}) generalizes \cite[Proposition 7]{CouvreurOrtho-Derivative} and \cite[Equation (10)]{MihailaThornburgh2026}. 
This also provides a proof of the already known fact (see \cite[Proposition 2]{Kyureghyan2007}) that the ortho-derivatives of crooked functions in an odd number of variables are invertible.
\begin{corollary}\label{cor:crooked-AB-invertibleOrtho}
    Assume $n$ is odd, and let $F \colon \F_2^n \to \F_2^n$ be a crooked function.
    Then $\pi_F$ is invertible.
\end{corollary}
\begin{proof}
    Since $n$ is odd and $F$ is plateaued, we know that $F$ is AB \cite[Proposition 163]{CarletBook}.
    A well-known characterization of AB functions is that they are those APN functions such that $\mult_{\graph F}(t,F(t)+v)=\frac{2^{n-1}-1}{3}$ for all $t,v \in \F_2^n$ where $v \neq0$, see \cite{vandamflass}.
    Applying \cref{eq:ortho-weight-mult}, we see that $v \cdot \pi_F$ is balanced (that is, has Hamming weight $2^{n-1}$) for all $v \neq 0$, i.e. $\pi_F$ is invertible.    
\end{proof}

We will soon apply \Cref{lem:ortho-cpts-Walsh2} to crooked functions, but we first require the following lemma in order to more easily understand the convolutional product of \Cref{lem:ortho-cpts-Walsh2}.

\begin{lemma}\label{lem:1S1S1_SandH_otimes}
    Let $S \subseteq \F_2^d$ be a Sidon set and let $H \subseteq \F_2^d$ be an affine hyperplane.
    For any $a \in \F_2^d$, we have
    \[
    (1_S \otimes 1_S \otimes 1_{S \cap H})(a)  
    = 
    \begin{cases}
        2\mult_S(a) + 4\mult_{S \cap H}(a) & a \in H \setminus S, \\
        |S| + 2(|S \cap H|-1) & a \in H \cap S, \\
        4\mult_S(a) - 4\mult_{S \cap H^c}(a) & a \in H^c \setminus S, \\
        2|S\cap H| & a \in H^c \cap S.
    \end{cases}
    \]
\end{lemma}
\begin{proof}
    Let $a \in \F_2^d$.
    Let $T_a = \set{(x,y,z) \in S^2 \times (S \cap H) : x+y+z=a}$.
    By \cref{eq:set-conv-equality}, we know that $(1_S \otimes 1_S \otimes 1_{S \cap H})(a)=|T_a|$.
    
    Let us first consider the case that $a \in H \setminus S$.
    If $(u,v,w) \in S^3$ such that $u+v+w=a$ (note that $u,v,w$ are necessarily distinct), then exactly one or three points of $(u,v,w)$ are in $H$, and so some permutation of $(u,v,w)$ is in $T_a$.
    So, $\mult_S(a) - \mult_{S \cap H}(a)$ is equal to the number of 3-sets $\set{u,v,w} \subseteq S$ intersecting $H$ in exactly one point and $u+v+w=a$.
    Since  
    \[
        T_a = \set{(x,y,z) \in (S \cap H)^3 : x+y+z=a} \cup \set{(x,y,z) \in (S \cap H^c)^2 \times (S \cap H) : x+y+z=a},
    \]
    and the first set on the right-hand-side has size $6\mult_{S 
    \cap H}(a)$, while the second has size $2(\mult_S(a) -\mult_{S \cap H}(a))$, we deduce that $|T_a| = 2\mult_S(a) + 4\mult_{S \cap H}(a)$.

    Now, assume $a \in S\cap H$.
    Then if $(u,v,w) \in S^3$ with $u+v+w=a$, then $u,v,$ and $w$ are not all distinct since $S$ is Sidon.
    So, for $(x,y,z) \in T_a$, there is at least one repeated element, and in particular, at least one of $x,y,z$ is equal to $a$.
    If $x=y$, then $z =a$, and there are $|S|$ such triples.
    If $x=z$, then $y=a$, and there are $| S \cap H|$ such triples.
    Similarly, if $y=z$, then $x=a$, and there are $|S \cap H|$ such triples.
    These three cases only intersect at $(x,y,z) = (a,a,a)$, and so $|T_a| = |S| + 2|S \cap H|-2$.

    Now, observe that by \cref{eq:set-conv-equality}, we have 
    \[
    (1_S \otimes 1_S \otimes 1_S)(a)=(1_S \otimes 1_S \otimes 1_{S \cap H})(a) +  (1_S \otimes 1_S \otimes 1_{S \cap H^c})(a).
    \]
    Recall from \Cref{prelim} that $(1_S \otimes 1_S \otimes 1_S)(a)$ equals $6\mult_S(a)$ if $a \notin S$, and $3|S|-2$ otherwise.
    The result then follows by applying the first two cases of this proof to $(1_S \otimes 1_S \otimes 1_{S \cap H^c})(a)$ when $a \in H^c \setminus S$ and $a \in H^c \cap S$.
    Indeed, if $a \in H^c \setminus S$, then 
    \[
    |T_a| = 6\mult_S(a) - (2\mult_S(a) + 4\mult_{S \cap H^c}(a)) = 4\mult_S(a) - 4\mult_{S \cap H^c}(a),
    \]
    and if $a \in H^c \cap S$, then 
    \[
    |T_a| = 3|S|-2 - (|S| + 2(|S \cap H^c|-1)) = 2|S| -2|S \cap H^c|=2|S \cap H|.
    \]
\end{proof}

We then have the following combinatorial description of the Walsh transform of the ortho-derivative of a crooked function.
\begin{theorem}\label{thm:crooked-function-ortho-Walsh}
    Let $F \colon \F_2^n \to \F_2^n$ be a crooked function, and let $u,v \in \F_2^n$ such that $v \neq 0$.
    Let $t \in \F_2^n$, let $H_{u,t} = t + \set{0,u}^\perp$, and let $F' = F|_{H_{u,t}}$.
    Then
    \[
    W_{\pi_F}(u,v) 
    = 
    \begin{cases}
       6\mult_{\graph F}(t,F(t)+v) +2 - 2^n& u=0 \\
       8\mult_{\graph{F'}}(t,F(t)+v) -2\mult_{\graph F}(t,F(t)+v) + 2 & u \neq 0.
    \end{cases}
    \]
\end{theorem}
\begin{proof}
    Applying \Cref{lem:ortho-cpts-Walsh2} with $A = \F_2^n$ and $\Delta = 2$, we have $W_{\pi_F}(0,v) = 6 \mult_{\graph F}(t,F(t)+v) + 2 - 2^n$ and
    \[
    W_{\pi_F}(u,v) 
    =
    2(1_{\graph F} \otimes 1_{\graph F} \otimes 1_{\graph{F'}})(t,F(t)+v) - 6\mult_{\graph F}(t,F(t)+v)+2.
    \]
    By \Cref{lem:1S1S1_SandH_otimes}, we have $(1_{\graph F} \otimes 1_{\graph F} \otimes 1_{\graph{F'}})(t,F(t)+v) = 2\mult_{\graph F}(t,F(t)+v) + 4 \mult_{\graph{F'}}(t,F(t)+v)$ because $(t,F(t)+v) \in H_{u,t} \times \F_2^n$.
    The statement immediately follows.
\end{proof}

The following theorem is a generalization of \Cref{thm:crooked-function-ortho-Walsh} along with a recursive relation.
\begin{theorem}\label{thm:ortho-Walsh-multiplicity-description}
    Let $\ell$ and $k$ be integers such that $1 \leq \ell \leq k$.
    Let $L_0 \supseteq \cdots \supseteq L_\ell$ be a sequence of nested linear subspaces of $\F_2^n$ such that $\dim(L_i) = k-i$ for all $0 \leq i \leq \ell$.
    Let $u_1, \dots, u_\ell \in \F_2^n$ such that $L_i = L_{i-1} \cap \set{0,u_i}^\perp$.
    Assume that $F$ is an $(n,m)$-function such that if $0 \leq i \leq \ell-1$, then $F$ has the crooked property and is differentially 2-uniform on $L_i$.
    Let $S_i = \graph{F} \cap (L_i \times \F_2^m)$, and let $k_{i,v} =\mult_{S_i}(0,F(0)+v)$ for some fixed nonzero $v \in \F_2^m$.
    If $1 \leq i \leq \ell$, then
    \begin{equation}\label{eq:ortho-Walsh-mult-recursive}
    W_{\pi_{i-1}^v}(u_i) = 8k_{i,v} - 2k_{i-1,v} +2
    \end{equation}
    and 
    \begin{equation}\label{eq:multSi-Walshcoeffs-soln}
    2(2^{2i} k_{i,v} - k_{0,v}) = \sum_{j=0}^{i-1} 2^{2j} (W_{\pi_j^v}(u_{j+1})-2),
    \end{equation}
    where $\pi_j^v := \pi_{F|_{L_j}}^v$ is defined as in \Cref{def:generalized-ortho} for all $0 \leq j \leq \ell-1$.
\end{theorem}
\begin{proof}
    Let $i$ be such that $1 \leq i \leq \ell$.
    Let $H_{u_i}=\set{0,u_i}^\perp \times \F_2^m$.
    By \Cref{lem:ortho-cpts-Walsh2}, we have that 
    \[
    W_{\pi_{i-1}^v}(u_i) = 2(1_{S_{i-1}} \otimes 1_{S_{i-1}} \otimes 1_{S_{i-1} \cap H_{u_i}})(0,F(0)+v) -6k_{i-1,v} + 2.
    \]
    Clearly, $(0,F(0)+v) \in (\set{0,u_i}^\perp \times \F_2^m) \setminus S_{i-1}$.
    Combining this with $S_i = S_{i-1} \cap H_{u_i}$, we then have by \Cref{lem:1S1S1_SandH_otimes} that $(1_{S_{i-1}} \otimes 1_{S_{i-1}} \otimes 1_{S_{i-1} \cap H_{u_i}})(0,F(0)+v) = 2k_{i-1,v}+4k_{i,v}$.
    Hence, $W_{\pi_{i-1}^v}(u_i)= 8k_{i,v} - 2k_{i-1,v}+2$, proving \cref{eq:ortho-Walsh-mult-recursive}.
    
    Upon rearrangement of the first statement, we have the recursive relation $k_{i,v} = \frac{k_{i-1,v}}{4} + \frac{1}{8}(W_{\pi_{i-1}^v}(u_i)-2)$.
    We will now prove, by induction, the identity 
    \begin{equation}\label{eq:recursive-rel}
    k_{i,v} = \frac{k_{0,v}}{2^{2i}} + \frac{1}{8} \sum_{j=0}^{i-1} 2^{2(j-i)} (W_{\pi_j^v}(u_{j+1})-2)
    \end{equation}
    for all $1 \leq i \leq \ell$. 
    From the recursive relation above, we have $k_{1,v} = \frac{k_{0,v}}{4}+\frac{1}{8}(W_{\pi_0^v}(u_1)-2)$, proving the base case.
    Now, assume $1 \leq i < \ell$ such that \cref{eq:recursive-rel} holds.
    Then 
    \begin{align*}
    k_{i+1,v} &= 
    \frac{k_{i,v}}{4} + \frac{1}{8}(W_{\pi_i^v}(u_{i+1})-2) \\
    &= \frac{1}{4} \parens{\frac{k_{0,v}}{2^{2i}} + \frac{1}{8} \sum_{j=0}^{i-1} 2^{2(j-i)} (W_{\pi_j^v}(u_{j+1})-2)}+\frac{1}{8}(W_{\pi_i^v}(u_{i+1})-2) \\
    &= \frac{k_{0,v}}{2^{2(i+1)}} + \frac{1}{8} \sum_{j=0}^i 2^{2(j-(i+1))} (W_{\pi_j^v}(u_{j+1})-2),
    \end{align*}
    as desired.    
    Multiplying both sides of \cref{eq:recursive-rel} by $2^{2i+1}$ and rearranging then yields \cref{eq:multSi-Walshcoeffs-soln}.
\end{proof}

We can also characterize when $\pi_{F|_L}^v$ is identically zero (we note that the following proof is almost the same as \cite[Proposition 7.2]{MihailaThornburgh2026}).

\begin{proposition}\label{prop:zerocpt-characterization}
    Let $F$ be an $(n,m)$-function that has the crooked property and is differentially $2$-uniform on a $k$-dimensional affine subspace $A \subseteq \F_2^n$, and let $v \in \F_2^m$ be nonzero.
    Then $\pi_{F|_A}^v$ is identically zero (equivalently $\pi_{F|_A}^v$ is affine) if and only if $k$ is even and for any $t \in A$, there exists a partition $\set{\set{x_i, y_i, z_i} : 1 \leq i \leq \frac{2^k-1}{3}}$ of $A \setminus \set{t}$ such that 
    \[
    (x_i + y_i + z_i, F(x_i)+F(y_i)+F(z_i))=(t,F(t)+v)
    \]
    for all $1 \leq i \leq \frac{2^k-1}{3}$.
    If so, then $\deg_{alg}(F|_A)\geq k$.
\end{proposition}
\begin{proof}
    Let $S = \graph{F|_A}$.
    By \cref{eq:ortho-weight-mult}, we know that $\pi_{F|_A}^v$ is identically zero if and only if $\mult_S(t,F(t)+v) = \frac{2^k-1}{3}$ for any $t \in A$.
    The claim is then immediate from the definition of a Sidon set and exclude multiplicity.
    For the final claim, observe that the sum $\sum_{x 
    \in A} F(x)$ is equal to $F(t) + \frac{2^k-1}{3}(F(t) + v) = v\neq 0$ if the above conditions hold.
\end{proof}

\section{On the algebraic-degree of the ortho-derivative of a crooked function and a proof of Gorodilova's conjecture}\label{sec:algdeg}

Recall from \Cref{conj:Gorodilova} that Gorodilova conjectured in 2020 that for any $n \geq 4$, any component of the ortho-derivative $\pi_F$ of a quadratic APN $(n,n)$-function $F$ has algebraic degree $n-2$.
In 2024, Couvreur, Canteaut, and Perrin in \cite{CouvreurOrtho-Derivative} showed that when $F$ is quadratic APN, the algebraic degree of $\pi_F$ is at most $n-2$ (that is, any component function of $\pi_F$ can have an algebraic degree at most $n-2$).
In this section, we prove \Cref{conj:Gorodilova} (in fact, we prove a more general result that proves \Cref{conj:Gorodilova} as a corollary), and we generalize the result of Couvreur, Canteaut, and Perrin to all crooked functions using \Cref{thm:crooked-function-ortho-Walsh}.

Recall that for a Sidon set $S \subseteq \F_2^d$ and a point $a \notin S$, we write $\mult_S(a)$ to denote the number of distinct triples $\set{x,y,z}$ contained in $S$ that sum to $a$.
We now introduce notation that we will use throughout the remainder of the paper.
For an $(n,m)$-function $F$ whose restriction to a linear subspace $L \subseteq \F_2^n$ has a Sidon set as its graph, that is, $F$ is differentially 2-uniform on $L$, we write $F' = F|_L$ and 
\[
m_L(v) = \mult_{\graph{F'}}(0,F(0)+v)
\]
for all nonzero $v \in \F_2^m$.
More generally, for any $u \in \F_2^n$, let
\[
m_{L,u}(v) = \mult_{\graph{F'} \cap (\set{0,u}^\perp \times \F_2^m)}(0,F(0)+v).
\]
In other words, $m_{L,u}(v)$ is the exclude multiplicity of $(0,F(0)+v)$ with respect to the graph of $F|_{L \cap \set{0,u}^\perp}$.
For the sake of simplicity, we let
\[
m(v) = m_{\F_2^n}(v) \quad \text{and} \quad
m_u(v) = m_{\F_2^n,u}(v).
\]
Note that $L \cap \set{0,u}^\perp$ has dimension $\dim(L)-1$ precisely when $u \notin L^\perp$, and so we will usually consider $u$ to be in a subspace $U$ satisfying $L^\perp \oplus U = \F_2^n$.
Furthermore, we have 
\begin{equation}\label{eq:mLuv-betaF}
3m_{L,u}(v) = \left|\set{\set{x,y} \subseteq L \cap \set{0,u}^\perp : \beta_F(x,y)=v}\right| 
\end{equation}
where $\beta_F(x,y)=F(0)+F(x)+F(y)+F(x+y)$ because any triple $\set{x,y,z} \subseteq L \cap \set{0,u}^\perp$ such that $(x+y+z, F(x)+F(y)+F(z))=(0,F(0)+v)$ is in one-to-one correspondence to the three distinct triples $\set{x,y}, \set{x,z}, \set{y,z}$ contained in the set from the right-hand-side of \cref{eq:mLuv-betaF}.

We will need the following technical (more or less known) lemma.

\begin{lemma}\label{lem:algdeg-lemma}
    Let $L \subseteq \F_2^n$ be a $d$-dimensional linear subspace, and let $U \subseteq \F_2^n$ be a subspace such that $L^\perp \oplus U = \F_2^n$.
    The algebraic degree of a nonzero Boolean function $f\colon L \to \F_2$ is the largest integer $k$ such that there exists a $k$-dimensional linear subspace $V \subseteq L$ and some $t_0 \in L$ such that 
    \[
    \sum_{u \in U \cap V^\perp}(-1)^{u \cdot t_0} W_f(u) \equiv 
    2^d - 2^{d-k+1} \pmod{2^{d-k+2}}.
    \]
\end{lemma}
\begin{proof}
The algebraic degree of $f$ is equal to the largest dimension of an affine subspace $A \subseteq L$ such that the restriction of $f$ to $A$ has odd weight.
Note that for a $k$-dimensional affine subspace $A \subseteq L$, we have $W_{f|_A}(0)=\sum_{x \in A}(-1)^{f(x)} = 2^k-2\wt(f|_A)$, and so $\wt(f|_A)$ is odd if and only if $W_{f|_A}(0) \equiv 2^k-2 \pmod 4$.
Let $V$ be the underlying vector space of $A$, that is $A=t_0+V$ with $t_0 \in L$.
Note that $U \cap V^\perp$ has dimension $d-k$ because $V^\perp = L^\perp \oplus (U \cap V^\perp)$.
So, $1_A(x) = 2^{k-d} \sum_{u \in U \cap V^\perp}(-1)^{u \cdot (x+t_0)}$ for $x \in L$.
We compute 
\begin{align*}
    \sum_{u \in U \cap V^\perp} (-1)^{u \cdot t_0} W_f(u) 
    &= \sum_{u \in U \cap V^\perp} (-1)^{u \cdot t_0} \sum_{x \in L} (-1)^{f(x) + u \cdot x} \\
    &= \sum_{x \in L}(-1)^{f(x)}
    \sum_{u \in U \cap V^\perp} (-1)^{u \cdot (x + t_0)} \\
    &= 2^{d-k} \sum_{x \in A} (-1)^{f(x)} \\
    &= 2^{d-k} W_{f|_A}(0).
\end{align*}
Thus, $W_{f|_A}(0)=2^{k-d}\sum_{u \in U \cap V^\perp} (-1)^{u \cdot t_0} W_f(u)$, and the statement follows because $W_{f|_A}(0) \equiv 2^k-2 \pmod 4$ if and only if $\sum_{u \in U \cap V^\perp}(-1)^{u \cdot t_0} W_f(u) \equiv 2^d-2^{d-k+1} \pmod{2^{d-k+2}}$.
\end{proof}

Recall that in \cite{CouvreurOrtho-Derivative}, it was shown that any component function of the ortho-derivative of a quadratic APN function has algebraic degree at most $n-2$. 
We will derive a more general result which will also generalize \cite[Theorem 1]{CouvreurOrtho-Derivative} to the case of the ortho-derivative of crooked functions.

\begin{proposition}\label{prop:crooked-orthoderivative-algdegbound}
    Let $n$ and $k$ be integers such that $3 \leq k \leq n$, and let $L \subseteq \F_2^n$ be a $k$-dimensional linear subspace.
    Let $F$ be an $(n,m)$-function that has the crooked property and is differentially 2-uniform on $L$.
    Let $S = \graph{F|_L}$, let $v \in \F_2^m$ be nonzero.
    Then $\deg_{alg}(\pi_{F|_L}^v)$ is at most $k-2$ (resp. equal to $k$) if and only if $\mult_S(0,F(0)+v)$ is odd (resp. $\mult_S(0,F(0)+v)$ is even).
    As a consequence, if $n =m$ and $F$ is a crooked function, then $\deg_{alg}(\pi_F) \leq n-2$.
\end{proposition}
\begin{proof}
    Let $U \subseteq \F_2^n$ be a linear subspace such that $L^\perp \oplus U = \F_2^n$.
    If $\pi_{F|_L}^v$ is identically zero, then $\mult_S(0,F(0)+v) = \frac{2^k-1}{3}$, which is odd, by \cref{eq:ortho-weight-mult}.
    So, assume throughout the remainder of this proof that $\pi_{F|_L}^v$ is not identically zero.
    By \Cref{lem:algdeg-lemma}, we know that $\pi_{F|_L}^v$ has algebraic degree at most $k-2$ if and only if there does not exist a $(k-1)$-dimensional linear subspace $V \subseteq L$ and $t_0 \in L$ such that 
    \[
    \sum_{u \in U \cap V^\perp}(-1)^{u \cdot t_0} W_{\pi_{F|_L}^v}(u) \equiv 2^k - 4 \equiv 4 \pmod 8.
    \]
    If $V \subseteq L$ is a subspace of dimension $k-1$, then $U \cap V^\perp = \set{0,u}$ for some nonzero $u \in U$, and on the other hand, every nonzero $u \in U$ corresponds to the $(k-1)$-dimensional linear subspace $L \cap \set{0,u}^\perp$ of $L$.
    Therefore, $\pi_{F|_L}^v$ has algebraic degree at most $k-2$ if and only if 
    \[
    W_{\pi_{F|_L}^v}(0) + (-1)^{u \cdot t_0} W_{\pi_{F|_L}^v}(u) \not\equiv 4 \pmod 8
    \]
    for all nonzero $u \in U$ and $t_0 \in L$.
    For a nonzero $u \in U$ and $t_0 \in L$, we have by \Cref{lem:ortho-cpts-Walsh2} and \Cref{thm:ortho-Walsh-multiplicity-description} that 
    \begin{align*}
        W_{\pi_{F|_L}^v}(0) + (-1)^{u \cdot t_0} W_{\pi_{F|_L}^v}(u)
        &= 6 m_L(v) + 2 - 2^k + (-1)^{u \cdot t_0} (-2m_L(v) + 8m_{L,u}(v) + 2) \\
        &= \begin{cases}
            4m_L(v) + 8m_{L,u}(v) +4-2^k & u \cdot t_0 =0, \\
            8m_L(v) -8m_{L,u}(v) -2^k & u \cdot t_0 =1.
        \end{cases}
    \end{align*}
    If $u \cdot t_0=1$, it is obvious that $W_{\pi_{F|_L}^v}(0) + (-1)^{u \cdot t_0} W_{\pi_{F|_L}^v}(u)$ is divisible by $8$ since $k \geq 3$.
    For the case that $u \cdot t_0=0$, first note  
    \[
    4m_L(v) + 8m_{L,u}(v) +4-2^k \equiv 4(m_L(v)+1) \pmod 8,
    \]
    which is congruent to $0$ modulo $8$ if and only if $m_L(v)$ is odd.
    This proves that $\deg_{alg}(\pi_{F|_L}^v) \leq k-2$ if and only if $m_L(v)$ is odd. 
    Since $\wt(\pi_{F|_L}^v) = 2^k-1-3m_L(v)$ holds by \cref{eq:ortho-weight-mult}, we have that if $\deg_{alg}(\pi_{F|_L}^v)>k-2$, then $\wt(\pi_{F|_L}^v)$ is odd, i.e. $\deg_{alg}(\pi_{F|_L}^v)=k$.
    
    Now, assume $n=m$ and that $F$ is a crooked function, and take $L = \F_2^n$.
    From \cref{eq:ortho-weight-mult}, we have $\mult_S(0,F(0)+v) = \frac{2^n-1-\wt(v \cdot \pi_F)}{3}$.
    If $\wt(v \cdot \pi_F)=0$, then $n$ is even as $\mult_S(0,F(0)+v) = \frac{2^n-1}{3}$.
    However, this is a contradiction by \Cref{prop:zerocpt-characterization} because $\deg_{alg}(F) < n$ as plateaued functions cannot have algebraic degree $n$ \cite{CarletBook}.
    Therefore, $v \cdot \pi_F$ is not identically zero, and the second statement then immediately follows from \Cref{prop:plateaued-oddmults}, which provides the fact that $\mult_S(0,F(0)+v)$ is odd.
\end{proof}

Thus, we have generalized  \cite[Theorem 1]{CouvreurOrtho-Derivative} by proving that for $n \geq 3$, a crooked function $F \colon \F_2^n \to \F_2^n$, and a point $v \in \F_2^n \setminus \set{0}$, we know that $v \cdot \pi_F$ has algebraic degree at most $n-2$.
In order to prove \Cref{conj:Gorodilova}, it remains to prove that for any quadratic APN function $F$, we have $\deg_{alg}(v \cdot \pi_F) \geq n-2$ for all $v \neq 0$.
We will characterize when equality occurs in \Cref{prop:crooked-orthoderivative-algdegbound} in terms of the following Boolean functions being of algebraic degree $2$.

\begin{definition}\label{def:exclude-parity}
    Let $F$ be an $(n,m)$-function that is differentially 2-uniform on a linear subspace $L \subseteq \F_2^n$, and let $U \subseteq \F_2^n$ be any linear subspace such that $L^\perp \oplus U = \F_2^n$.
    For any nonzero $v \in \F_2^m \setminus \set{0}$, we call the function $f_{L,U,v} \colon U \to \F_2$ defined by 
    \[
    f_{L,U,v}(u) = (m_{L,u}(v) \mod 2)
    \]
    an \textit{exclude parity function of $F$ at $(0,F(0)+v)$.}
    If $L =U=\F_2^n$, then we simply denote $f_{L,U,v}(u)$ as $f_v(u)$, and we have 
    \[
    f_v(u) = (m_u(v) \mod 2).
    \]
\end{definition}

Note that under the same assumptions and notation as \Cref{def:exclude-parity}, by \Cref{eq:mLuv-betaF}, we have 
\[
f_{L,U,v}(u) \equiv \left|\set{\set{x,y} \subseteq L \cap \set{0,u}^\perp : \beta_F(x,y)=v}\right|  \pmod{2}.
\]

\begin{remark}
    Under the same assumptions and notation as in \Cref{def:exclude-parity}, clearly the definition of $f_{L,U,v}$ depends on $U$, but only up to a linear change in variables.
    Indeed, let $U_1, U_2 \subseteq \F_2^n$ be subspaces such that $L^\perp \oplus U_1 = \F_2^n = L^\perp \oplus U_2$.
    There exists a linear isomorphism $\varphi \colon U_1 \to U_2$ such that $u+\varphi(u) \in L^\perp$ for all $u \in U_1$.
    This implies $u \cdot x = \varphi(u)\cdot x$ for all $x \in L$, implying $L \cap \set{0,u}^\perp = L \cap \set{0,\varphi(u)}^\perp$.
    For any nonzero $v \in \F_2^m$, we then have $m_{L,u}(v) = m_{L,\varphi(u)}(v)$ for all $u \in U_1$, implying $f_{L,U_1,v}(u) = f_{L,U_2, v}(\varphi(u))$.
    That is, the choice of $U$ in \Cref{def:exclude-parity} only changes the function up to linear equivalence.
\end{remark}

Let us first prove that any exclude parity function is quadratic (recall functions having an algebraic degree at most $2$ are quadratic, rather than having algebraic degree equal to $2$).

\begin{proposition}\label{prop:parity-fcn-algdeg-bound}
    Let $F$ be an $(n,m)$-function that is differentially $\set{0,2}$-valued on a linear subspace $L \subseteq \F_2^n$, and let $U \subseteq \F_2^n$ be a linear subspace such that $L^\perp \oplus U = \F_2^n$.
    Then $f_{L,U,v}$ is quadratic for all nonzero $v \in \F_2^m$.
\end{proposition}
\begin{proof}
    Let $F' = F|_L$.
    For any $u \in U$, we have by definition that $m_{L,u}(v)$ equals
    \[
    |\{\set{x,y,z} \subseteq L \cap \set{0,u}^\perp : (x+y+z, F(x)+F(y)+F(z))=(0,F(0)+v)\}|.
    \]
    Let $\mathcal{S}_{L,v}$ be the set of $2$-dimensional linear subspaces $E \subseteq L$ such that $\sum_{x \in E} F(x)=v$.
    Then $m_{L,u}(v) = \sum_{E \in \mathcal{S}_{L,v}} 1_{E^\perp}(u)$ for any $u \in U$.
    Therefore,
    \begin{equation}\label{eq:fvu-subspace-sum}
    f_{L,U,v}(u) = \bigoplus_{E \in \mathcal{S}_{L,v}} 1_{E^\perp}(u).
    \end{equation}
    Hence, $f_{L,U,v}$ is a sum of Boolean functions of algebraic degree $2$ (we consider the restrictions of the Boolean functions $1_{E^\perp}$, where $E \in \mathcal{S}_{L,v}$, to $U$), and so $\deg_{alg}(f_{L,U,v}) \leq 2$.
\end{proof}

\begin{remark}
    Let $F$ be an $(n,m)$-function, let $L \subseteq \F_2^n$ be a subspace, and let $\mathcal{S}_{L,v}$ be defined as in the proof of \Cref{prop:parity-fcn-algdeg-bound}.
    Let $F' = F|_L$.
    If $v = 0$, then $\mathcal{S}_{L,v}$ is empty when $\graph{F'}$ is Sidon.
    If $v \neq 0$, then $\mathcal{S}_{L,v}$ being non-empty is equivalent to $F|_L$ having the so-called D-property \cite{Carlet_apnGraphMaximal,Taniguchi2023}.
    Consider when $F|_L$ is plateaued.
    Then, the size of the set $\set{(a,b) \in L^2 : D_a D_b F(x) = w}$ does not depend on $x \in L$ for any $w \in \F_2^m$ \cite{CarletPlateaued}, and this is equivalent to $\widehat{W_{F|_L}^3}(x,F(x)+w)$ not depending on $x \in L$ for all $w \in \F_2^m$ (cf. \cite[Section 4.2]{MihailaThornburgh2026}).
    It is straightforward to show that $\graph{F'}$ is Sidon if and only if 
    \[
    \widehat{W_{F|_L}^3}(x,F(x))
    = 2^{n+m}|\set{(a,b,c) \in L^3 : (a+b+c, F(a)+F(b)+F(c))=(x,F(x))}|
    \]
    is equal to $2^{n+m}(3|L|-2)$ for all $x \in L$ because any such $(a,b,c) \in L^3$ cannot consist of three distinct points.
    So, if $F|_L$ is plateaued, then $\graph{F|_L}$ is a Sidon set if and only if $\mathcal{S}_{L,0}$ is empty, which is equivalent to $\widehat{W_{F|_L}^3}(0,F(0))=2^{n+m}(3|L|-2)$.
\end{remark}

We now describe when equality is obtained in \Cref{prop:crooked-orthoderivative-algdegbound} in terms of exclude parity functions having algebraic degree $2$.

\begin{proposition}\label{prop:orthoderiv-algdeg-equiv}
    Assume $n\geq k \geq 4$, and let $L \subseteq \F_2^n$ be a $k$-dimensional linear subspace, and let $F$ be an $(n,m)$-function that has the crooked property and is differentially 2-uniform on $L$.
    Let $U \subseteq \F_2^n$ be a subspace such that $L^\perp \oplus U = \F_2^n$, let $S = \graph{F|_L}$, and let $v \in \F_2^m \setminus \set{0}$.
    The following are equivalent: 
    \begin{enumerate}
        \item $\deg_{alg}(\pi_{F|_L}^v) = k-2$;
        \item $m_L(v)$ is odd and there exist linearly independent $e_1, e_2 \in U$ such that $m_{L,e_1}(v) + m_{L,e_2}(v) + m_{L,e_1+e_2}(v)$ is even;
        \item $m_L(v)$ is odd and $\deg_{alg}(f_{L,U,v})=2$.
    \end{enumerate}
\end{proposition}
\begin{proof}
    By \Cref{prop:crooked-orthoderivative-algdegbound}, we know that $\deg_{alg}(\pi_{F|_L}^v)\leq k-2$ if and only if $m_L(v)$ is odd, so the equality $\deg_{alg}(\pi_{F|_L}^v)=k-2$ implies $m_L(v)$ is odd.
    So, if $m_L(v)$ is odd, then we know that $\deg_{alg}(\pi_{F|_L}^v) = k-2$ if and only if there exists a $(k-2)$-dimensional subspace $V \subseteq L$ and $t_0 \in L$ such that
    \[
    \sum_{u \in U \cap V^\perp} (-1)^{u \cdot t_0} W_{\pi_{F|_L}^v}(u) 
    \equiv 2^k - 2^{k-(k-2)+1} 
    \equiv 8
    \pmod{16}
    \]
    by \Cref{lem:algdeg-lemma}.

    Let $V \subseteq L$ be a $(k-2)$-dimensional subspace, and let $t_0 \in L$.
    Then $U \cap V^\perp$ is a 2-dimensional subspace and equals $\langle e_1, e_2 \rangle$ for some linearly independent $e_1,e_2 \in U$.
    The size of the intersection $\set{0,t_0}^\perp$ with $U \cap V^\perp$ is either $2$ or $4$.
    In the former case, we have, upon assuming without loss of generality that $e_1 \cdot t_0 = 0$, that 
    \begin{align*}
     \sum_{u \in U \cap V^\perp} (-1)^{u \cdot t_0} W_{\pi_{F|_L}^v}(u) 
     &= W_{\pi_{F|_L}^v}(0) +  W_{\pi_{F|_L}^v}(e_1) -  W_{\pi_{F|_L}^v}(e_2) -W_{\pi_{F|_L}^v}(e_1+e_2) \\
     &= 8m_L(v) - 2^k + 8(m_{L,e_1}(v) - m_{L,e_2}(v) - m_{L,e_1+e_2}(v))
    \end{align*}
    by \Cref{lem:ortho-cpts-Walsh2} and \Cref{thm:ortho-Walsh-multiplicity-description}.
    In the latter case, we have 
    \begin{align*}
        \sum_{u \in U \cap V^\perp} (-1)^{u \cdot t_0} W_{\pi_{F|_L}^v}(u) 
        &=  W_{\pi_{F|_L}^v}(0) +  W_{\pi_{F|_L}^v}(e_1) +  W_{\pi_{F|_L}^v}(e_2) +  W_{\pi_{F|_L}^v}(e_1+e_2) \\
        &=8 + 8(m_{L,e_1}(v) + m_{L,e_2}(v) + m_{L,e_1+e_2}(v))-2^k.
    \end{align*}
    In the first case, we see that if $m_L(v)$ is odd, then $\sum_{u \in U \cap V^\perp} (-1)^{u \cdot t_0} W_{\pi_{F|_L}^v}(u) \equiv 8 \pmod{16}$ is equivalent to $m_{L,e_1}(v) + m_{L,e_2}(v) + m_{L,e_1+e_2}(v)$ being even, i.e. $f_{L,U,v}(e_1) \oplus f_{L,U,v}(e_2) \oplus f_{L,U,v}(e_1\oplus e_2) = 0$.
    The equivalence of $\sum_{u \in U \cap V^\perp} (-1)^{u \cdot t_0} W_{\pi_{F|_L}^v}(u) \equiv 8 \pmod{16}$ and $m_{L,e_1}(v) + m_{L,e_2}(v) + m_{L,e_1+e_2}(v)$ being even is also straightforward in the second case.
    The equivalence of the first and second statements then holds.

    Assume $m_L(v)$ is odd, i.e. $f_{L,U,v}(0)=1$.
    Since $\deg_{alg}(f_{L,U,v}) \leq 2$ by \Cref{prop:parity-fcn-algdeg-bound}, the condition $\deg_{alg}(f_{L,U,v}) = 2$ is equivalent to the existence of linearly independent $e_1, e_2 \in U$ such that $D_{e_1}D_{e_2}f_{L,U,v}(0)=\bigoplus_{u \in \langle e_1, e_2 \rangle}f_{L,U,v}(u) = 1$.
    However, $\bigoplus_{u \in \langle e_1, e_2 \rangle}f_{L,U,v}(u) = 1$ if and only if $m_{L,e_1}(v) + m_{L,e_2}(v) + m_{L,e_1+e_2}(v)$ is even, so the second and third statements are equivalent.
\end{proof}

We now prove the following result, which implies Gorodilova's conjecture from \cite{Gor20} that the ortho-derivative of any quadratic APN $(n,n)$-function has all components of algebraic degree $n-2$, where $n \geq 4$.

\begin{theorem}\label{thm:quadraticAPN-orthoderiv-algdeg}
    Assume $n\geq k \geq 4$, let $L \subseteq \F_2^n$ be a $k$-dimensional linear subspace, and let $F$ be an $(n,m)$-function with the crooked property that is differentially 2-uniform on $L$.
    Let $v \in \F_2^m$ be nonzero.
    If $m_L(v)$ is odd and there exists $c \in \F_2^m$ such that $c \cdot v = 1$ and $c \cdot F|_L$ is quadratic, then $\deg_{alg}(\pi_{F|_L}^v)=k-2$.
\end{theorem}
\begin{proof}
    Let $v \in \F_2^m$ be nonzero, and assume that $m_L(v)$ is odd and there exists a nonzero $c \in \F_2^m$ such that $c \cdot v = 1$ and $c \cdot F|_L$ is quadratic.
    Let $\mathcal{S}_{L,v}$ be the set of 2-dimensional linear subspaces $E \subseteq L$ such that $\sum_{x \in E} F(x) = v$.
    For any $E \in \mathcal{S}_{L,v}$, let $a_E, b_E \in L$ constitute a basis for $E$.
    Let $U \subseteq \F_2^n$ be a subspace such that $L^\perp \oplus U = \F_2^n$.
    Let $\ell_1, \dots, \ell_k\in L$ be a basis for $L$, and let $u_1, \dots, u_k \in U$ be the unique basis for $U$ such that $\ell_i \cdot u_j = \delta_{ij}$ where $\delta_{ij}$ is the Kronecker delta.
    Let $a_E = \sum_{i=1}^k a_{E,i} \ell_i$ and $b_E = \sum_{i=1}^k b_{E,i} \ell_i$, where $a_{E,i}, b_{E,i} \in \F_2$.

    Let $u = \sum_{i=1}^k y_i u_i \in U$, where $y_i \in \F_2$.
    For any $E \in \mathcal{S}_{L,v}$, it is clear that $1_{E^\perp}(u) = (a_E\cdot u \oplus 1)(b_E \cdot u \oplus 1)$.
    Observe that 
    \begin{align*}
        a_E \cdot u &= \parens{ \sum_{i=1}^k a_{E,i} \ell_i} \cdot \parens{ \sum_{j=1}^k y_j u_j } 
        = \bigoplus_{i,j=1}^k a_{E,i}y_j (\ell_i \cdot u_j) 
        = \bigoplus_{i=1}^k a_{E,i}y_i,
    \end{align*}
    and similarly we have $b_E \cdot u = \bigoplus_{j=1}^k b_{E,j}y_j$.
    Then 
    \begin{align*}
        (a_E \cdot u)(b_E \cdot u) 
        &= \bigoplus_{i,j=1}^k a_{E,i}b_{E,j}y_i y_j 
        =\bigoplus_{i=1}^k a_{E,i}b_{E,i}y_i \oplus \bigoplus_{1 \leq i < j \leq k} \alpha_{E,i,j} y_i y_j,
    \end{align*}
    where we let $\alpha_{E,i,j}=a_{E,i} b_{E,j} \oplus a_{E,j} b_{E,i}$.
    
    Let $q_{L,U,v} \colon U \to \F_2$ be the homogeneous quadratic Boolean function whose ANF only consists of the degree 2 monomials that are in the ANF of $f_{L,U,v}$ (recall that $f_{L,U,v}$ is the Boolean function on $U$ taking value $1$ at $u$ if and only if $m_{L,u}(v)=\mult_{\graph{F'}}(0,F(0)+v)$ is odd, where $F' = F|_{L \cap \set{0,u}^\perp}$).
    Recall from \cref{eq:fvu-subspace-sum} that $f_{L,U,v}(u) = \bigoplus_{E \in \mathcal{S}_{L,v}} 1_{E^\perp}(u)$, and so 
    \begin{equation}\label{qv-anf}
    q_{L,U,v}(u) 
    = \bigoplus_{E \in \mathcal{S}_{L,v}} \bigoplus_{1 \leq i < j 
    \leq k} \alpha_{E,i,j} y_i y_j = \bigoplus_{1 \leq i < j\leq k} y_i y_j \bigoplus_{E \in \mathcal{S}_{L,v}} \alpha_{E,i,j}.
    \end{equation}
    Since $c\cdot F|_L$ is quadratic, the function $\beta_{c \cdot F}(a,b) = c\cdot (F(0)+F(a)+F(b)+F(a+b))$ is bilinear on $L^2$.
    For $E \in \mathcal{S}_{L,v}$, we have
    \[
    1 = \beta_{c\cdot F}(a_E, b_E) 
    = \beta_{c \cdot F}\parens{ \sum_{i=1}^k a_{E,i}\ell_i, \sum_{j=1}^k b_{E,j}\ell_j} 
    = \bigoplus_{i,j=1}^k a_{E,i}b_{E,j} \beta_{c \cdot F}(\ell_i, \ell_j),
    \]
    which is equal to $\bigoplus_{1 \leq i< j \leq k} \alpha_{E,i,j} \beta_{c \cdot F}(\ell_i, \ell_j)$ since $\beta_{c \cdot F}$ is symmetric and $\beta_{c \cdot F}(\ell_i, \ell_i) = 0$ for all $1 \leq i \leq k$.
    Since $|\mathcal{S}_{L,v}| = m_L(v)$ is odd, we have
    \begin{align*}
        1  
        &= \bigoplus_{E \in \mathcal{S}_{L,v}} \bigoplus_{ 1 \leq i < j \leq k} \alpha_{E,i,j} \beta_{c \cdot F}(\ell_i, \ell_j) 
        = \bigoplus_{ 1 \leq i < j \leq k}\beta_{c \cdot F}(\ell_i, \ell_j) 
        \bigoplus_{E \in \mathcal{S}_{L,v}}
        \alpha_{E,i,j}
    \end{align*}
    If $q_{L,U,v}$ is identically zero, then \cref{qv-anf} implies $\bigoplus_{E \in \mathcal{S}_{L,v}}  \alpha_{E,i,j} = 0$ for all $1 \leq i < j \leq k$ because the ANF of $q_{L,U,v}$ is unique.
    However, this is impossible as $0 \neq 1=
    \bigoplus_{ 1 \leq i < j \leq k}\beta_{c \cdot F}(\ell_i, \ell_j) 
        \bigoplus_{E \in \mathcal{S}_{L,v}}
        \alpha_{E,i,j}$.
    Hence, $q_{L,U,v}$ is not identically zero, or equivalently $\deg_{alg}(f_{L,U,v})=2$ by \Cref{prop:parity-fcn-algdeg-bound}.
    The result then follows by \Cref{prop:orthoderiv-algdeg-equiv}.
\end{proof}

We then have the following corollary, proving \Cref{conj:Gorodilova}.

\begin{corollary}\label{cor:more-quadratic-cpts-closer-to-conjecture}
    Assume $n \geq 4$, and let $F \colon \F_2^n \to \F_2^n$ be a crooked function.
    Let $\mathcal{Q}(F)$ be the subspace 
    \[
    \mathcal{Q}(F) = \set{b \in \F_2^n : b \cdot F \text{ is quadratic}}.
    \]
    Then $\pi_F$ has at least $2^n-2^{n-\dim(\mathcal{Q}(F))}$ components of algebraic degree $n-2$.
    In particular, $\deg_{alg}(v \cdot \pi_F) = n-2$ for all nonzero $v \in \F_2^n$ when $F$ is quadratic.
\end{corollary}
\begin{proof}
    By \Cref{prop:plateaued-oddmults}, we know that $\mult_{\graph F}(a,b)$ is odd for all $(a,b) \notin \graph{F}$, and so $m(v)$ is odd for all nonzero $v \in \F_2^n$.
    We know that if $v \notin \mathcal{Q}(F)^\perp$, then there exists a nonzero $c\in \F_2^n$ such that $c \cdot v = 1$, implying $v \cdot \pi_F$ has algebraic degree $n-2$ by \Cref{thm:quadraticAPN-orthoderiv-algdeg}.
    Thus, the number of components of $\pi_F$ of algebraic degree $n-2$ is at least $2^n - 2^{\dim(\mathcal{Q}(F)^\perp)}$, and the result follows.
\end{proof}

\subsection{Crooked functions in an even number of variables}

For $n$ even, an $n$-variable Boolean function is called \textit{semi-bent} if $f$ has linearity $2^{\frac{n}{2}+1}$.
The following result immediately implies that any quadratic APN function over $\F_2^n$, with $n$ even, has a semi-bent component function, but we will derive a lower bound on the number of semi-bent components of a quadratic APN function in \Cref{cor:n-semi-bent}.

We now provide another characterization of when $v \cdot \pi_F$ has algebraic degree equal to $n-2$ when $F$ is crooked.
For an even positive integer $n$ and a vectorial Boolean function $F$, denote by $\bentcomps{F}$ (resp. $\semibentcomps{F}$) the set of $b \in \F_2^n$ such that $b \cdot F$ is bent (resp. semi-bent).

\begin{proposition}\label{prop:algdeg-characterization-semibent}
    Assume $n\geq4$ is even.
    Let $F \colon \F_2^n \to \F_2^n$ be a crooked function.
    For any $b \in \F_2^n$ and any affine hyperplane $H \subseteq \F_2^n$, let $\lambda_{b,H}$ be the amplitude of $b\cdot F|_H$, and let $\lambda_{b,H}^2 = 2^{n+k_{b,H}}$.
    Let $e_1, e_2 \in \F_2^n$ be linearly independent vectors, and let $v \in \F_2^n$ be nonzero. 
    Then $m_{e_1}(v) + m_{e_2}(v)+m_{e_1+e_2}(v)$ is even if and only if 
    \[
    \sum_{b \in \semibentcomps{F}} (-1)^{v \cdot b} (2^{k_{b,H_1}} + 2^{k_{b,H_2}} + 2^{k_{b,H_3}}) \equiv 6 \pmod {12}
    \]
    where $H_1 = \set{0,e_1}^\perp, H_2 = \set{0,e_2}^\perp$ and $H_3 = \set{0,e_1+ e_2}^\perp$.
\end{proposition}
\begin{proof}
    Recall from \Cref{prelim} that the authors proved in \cite{CarletThornburghRestrictions} that the restriction of any partially-bent Boolean function to an affine hyperplane is plateaued.
    Then, any strongly plateaued function has a plateaued restriction on any affine hyperplane, so $\lambda_{b,H}$ is well-defined.
    Letting $S = \graph{F} \cap (H \times \F_2^n)$, recall that \Cref{rem:plateaued-restriction-mults} provides 
    $
     \mult_S(x,y)=\frac{1_H(x)}{6 \cdot 2^n} \sum_{b \in \F_2^n} (-1)^{b \cdot (y+ F(x))} \lambda_{b,H}^2
    $
    for all $(x,y) \notin S$.
    
    Since the dimension of an affine hyperplane $H\subseteq\F_2^n$ is odd, we know that $\lambda_{b,H} \geq 2^{\frac{\dim(H)+1}{2}} = 2^{\frac n2}$ for any $b \in \F_2^n$, implying $k_{b,H} \geq 0$.
    Let $e_1, e_2 \in \F_2^n$ be linearly independent vectors, and let $H_1, H_2$ and $H_3$ be defined as above.
    Then for any nonzero $v \in \F_2^n$, we have
    \begin{align*}
    6 (m_{e_1}(v) + m_{e_2}(v)+m_{e_1+e_2}(v)) 
    &=\sum_{b \in \F_2^n} (-1)^{v \cdot b} (2^{k_{b,H_1}} + 2^{k_{b,H_2}} + 2^{k_{b,H_3}}).
    \end{align*}
    It is well-known that the restriction of a bent function to an affine hyperplane is plateaued with linearity $2^{\frac{n}{2}}$ (see \cite{CanteautCharpinDecomposing,CarletThornburghRestrictions}), implying $k_{b,H_i} = 0$ for any $b \in \bentcomps{F}$ and $i \in \set{1, 2,3}$.
    Then, $6 (m_{e_1}(v) + m_{e_2}(v)+m_{e_1+e_2}(v))$ equals
    \begin{equation}\label{eq:3-linearities-componentsF}
    3 \cdot \widehat{1_{\bentcomps{F}}}(v) + \sum_{b \in \F_2^n \setminus \bentcomps{F}} (-1)^{v \cdot b} (2^{k_{b,H_1}} + 2^{k_{b,H_2}} + 2^{k_{b,H_3}}),
    \end{equation}
    and it was shown in the proof of \cite[Proposition 4.14]{MihailaThornburgh2026} that any value in the image set of $\widehat{1_{\bentcomps{F}}}$ is congruent to $2$ modulo $4$.
    In particular, $3\cdot \widehat{1_{\bentcomps{F}}}(v) \equiv 6  \pmod {12}$.
    Therefore, $m_{e_1}(v) + m_{e_2}(v)+m_{e_1+e_2}(v)$ is even if and only if
    $\sum_{b \in \F_2^n \setminus \bentcomps{F}} (-1)^{v \cdot b} (2^{k_{b,H_1}} + 2^{k_{b,H_2}} + 2^{k_{b,H_3}}) \equiv 6 \pmod{12}$.
    
    Now, consider the case that $b\cdot F$ is not bent nor semi-bent, that is, the linearity of $b \cdot F$ is at least $2^{\frac{n}{2}+2}$.
    Then for any affine hyperplane $H$, we know $\lambda_{b,H} \geq 2^{\frac{n}{2}+1}$ because  a restriction of a plateaued function, of amplitude $\lambda$, to an affine hyperplane has amplitude equal to $\lambda$ or $\frac{\lambda}{2}$, see \cite{CarletThornburghRestrictions}.
    Hence, $k_{b,H} \geq 2$.
    This implies that $2^{k_{b,H_1}} + 2^{k_{b,H_2}} + 2^{k_{b,H_3}}$ is divisible by $4$, and moreover, since $k_{b,H_i}$ is even for all $i \in \set{1, 2,3}$, we have that $2^{k_{b,H_1}} + 2^{k_{b,H_2}} + 2^{k_{b,H_3}}$ is divisible by $3$.
    So, $2^{k_{b,H_1}} + 2^{k_{b,H_2}} + 2^{k_{b,H_3}}$ is divisible by $12$.
    The claim of the statement immediately follows.
\end{proof}

\subsection{Crooked functions in an odd number of variables}

\subsubsection{A relation between the values of $W_{\pi_F}$ and where bent and semi-bent component functions of $F|_H$ are located}
In this paragraph, it is shown that when $n$ is odd, the ortho-derivative of a crooked $(n,n)$-function $F$ is related to the bent components of $F$ restricted to a hyperplane.
Let $H= \set{0,e}^\perp$ for some nonzero $e \in \F_2^n$.
Let $f \colon \F_2^n \to \F_2$ be a partially-bent function, and recall the \textit{linear kernel} of $f$ is the linear subspace 
$
\mathcal{E}_f = \set{a \in \F_2^n : D_a f \text{ is constant}}.
$
It was shown in \cite[Remark 2]{CarletThornburghRestrictions} that $f|_H$ and $f|_{H^c}$ are plateaued with amplitude $\lambda$ (resp. $\frac{\lambda}{2}$) if $e \in \mathcal{E}_f^\perp$ (resp. $e \notin \mathcal{E}_f^\perp$).

\begin{proposition}\label{prop:bentcomps-crookedAB-Walsh}
    Let $n\geq3$ be odd, let $F \colon\F_2^n \to\F_2^n$ be a crooked function, and let $H \subseteq \F_2^n$ be an affine hyperplane with underlying vector space $\set{0,e}^\perp$.
    Then $F|_H$ has exactly $2^{n-1}$ bent components and $2^{n-1}-1$ semi-bent components.
    Moreover, for any $v \in \F_2^n \setminus \set{0}$, we have
     \[
     |\bentcomps{F|_H} \cap \set{0,v}^\perp|= 2^{n-2} - \frac{1}{4} W_{\pi_F}(e,v).
     \]
\end{proposition}
\begin{proof}
From above, we know
$
\bentcomps{F|_H} = \set{b \in \F_2^n \setminus \set{0} : e \notin \mathcal{E}_{b \cdot F}^\perp }
$
since all component functions of $F$ have amplitude $2^{\frac{n+1}{2}}$.
Moreover, any non-bent component of $F|_H$ is semi-bent since $F$ is AB.
For any $b \neq0$, the linear kernel of $b \cdot F$ is equal to $\set{0,\pi_F^{-1}(b)}$ \cite[Proposition 2]{Kyureghyan2007}.
Hence, $\bentcomps{F|_H}=\supp(e \cdot \pi_F^{-1})$, which has size $\wt(e \cdot \pi_F^{-1}) = 2^{n-1}$ because $\pi_F^{-1}$ is a bijection by \Cref{cor:crooked-AB-invertibleOrtho}.
So, $F|_H$ has exactly $2^{n-1}$ bent components and $2^{n-1}-1$ semi-bent components.
Moreover, 
\begin{align*}
    |\bentcomps{F|_H} \cap \set{0,v}^\perp| &= \frac{|\supp(e \cdot \pi_F^{-1})|}{2} - \frac{W_{\pi_F^{-1}}(v,e)}{4} \\
    &= 2^{n-2} -\frac{W_{\pi_F^{-1}}(0,e) + W_{\pi_F^{-1}}(v,e)}{4} \\
    &= 2^{n-2} - \frac{1}{4} W_{\pi_F}(e,v).
\end{align*}
\end{proof}

We also establish the following facts on crooked AB functions.

\begin{proposition}
    Assume $n\geq3$ is odd.
    Let $F \colon \F_2^n \to \F_2^n$ be a crooked function, let $H \subseteq \F_2^n$ be an affine hyperplane, and let $S =\graph{F} \cap (H \times \F_2^n)$.
    Let $t \in H$ and let $v \in \F_2^n \setminus \set{0}$.
    Then $|\bentcomps{F|_H} \cap \set{0,v}^\perp| \leq \frac{2^n-2}{3}$ and $\frac{1}{2}|\bentcomps{F|_H} \cap \set{0,v}^\perp|$ has the opposite parity of $\mult_S(t,F(t)+v)$.
\end{proposition}
\begin{proof}
    Let $\set{0,e}^\perp$ be the underlying vector space of $H$, and let $v \in \F_2^n \setminus \set{0}$.
    From \Cref{prop:bentcomps-crookedAB-Walsh}, we have
    \begin{align*}
         |\bentcomps{F|_H} \cap \set{0,v}^\perp|
         &= 2^{n-2} - \frac{1}{4} W_{\pi_F}(e,v)\\
         &=  2^{n-2} - \frac{1}{4}  \parens{8\mult_S(t,F(t)+v) - 2 \cdot \frac{2^n-2}{6}+2} \\
         &= \frac{2^n-2}{3} -  2\mult_S(t,F(t)+v),
    \end{align*}
    where the second equality holds by \Cref{thm:crooked-function-ortho-Walsh} and the plateauedness of crooked functions on hyperplanes (the latter condition allowing us to consider an arbitrary point $t \in H$, see \Cref{rem:plateaued-restriction-mults}).
    The claim that $|\bentcomps{F|_H} \cap \set{0,v}^\perp| \leq \frac{2^n-2}{3}$ also follows. 
    Since $\frac{1}{2}|\bentcomps{F|_H} \cap \set{0,v}^\perp|= \frac{2^n-2}{6} -  \mult_S(t,F(t)+v)$ and $\frac{2^n-2}{6}$ is odd, the last claim is also immediate.
\end{proof}

\subsubsection{A connection to a result of \cite{Abbondati2024}}\label{Piccione}
In this paragraph, we characterize the strong D-property (introduced in \cite{CarletPiccioneStrongDProperty}) of crooked functions and we observe a connection with \cite{Abbondati2024}.
Let us first introduce the following background.
An $(n,m)$-function $F$ is said to have the \textit{D-property} if 
\[
\set{F(x)+F(y)+F(z)+F(x+y+z) : x,y,z \in \F_2^n} = \F_2^m,
\]
or equivalently, the sums of $F$ over all 2-dimensional affine subspaces cover $\F_2^m \setminus \set{0}$.
The D-property is also equivalent to the condition that for all $c \in \F_2^m$, there exists $t \in \F_2^n$ such that three (not necessarily distinct) points in $\graph{F}$ have sum equal to $(t,F(t)+c)$.
Moreover, we say that $F$ has the \textit{strong D-property} if for any affine hyperplane, we have:
\[
\set{F(x)+F(y)+F(z)+F(x+y+z) : x,y,z \in H} = \F_2^m,
\]
and this definition was first introduced in \cite{CarletPiccioneStrongDProperty}. 
As with the D-property, we can translate this property as the fact that for every affine hyperplane and every $c \in \F_2^m$, there exists $t \in H$ such that three (not necessarily distinct) points in $\graph{F|_H}$ have sum equal to $(t,F(t)+c)$.
Assuming that $n=m$ and $F$ is crooked, we know from \Cref{thm:crooked-function-ortho-Walsh} that for any affine hyperplane $H$ and nonzero $c \in \F_2^n$, the restriction $F'=F|_H$ has the property that the value of $\mult_{\graph{F'}}(t,F(t)+c)$ does not depend on the value of $t \in H$.
Hence, we know that $F$ has the strong D-property if and only if the graph of $F|_H$ is a maximal Sidon set in $H \times \F_2^n$ for all affine hyperplanes $H$.
By \Cref{rem:plateaued-restriction-mults}, we then know that $F$ has the strong D-property if and only if for any affine hyperplane $H$, we have for all nonzero $c \in \F_2^n$ that $\sum_{v \in \F_2^n} (-1)^{v \cdot c} \lambda_{v,H}^2 > 0$ where $\lambda_{v,H}$ is the amplitude of $v \cdot F|_H$.
Fix $H$ to be some affine hyperplane.
Then, by \Cref{rem:amplitude-sums-Sidon-iff}, we know that $\sum_{v \in \F_2^n}\lambda_{v,H}^2 = 2^n (3 \cdot 2^{n-1}-2)$, implying
\[
\sum_{v \in \F_2^n} (-1)^{v \cdot c} \lambda_{v,H}^2 
= 2^n(3 \cdot 2^{n-1}-2) - 2\sum_{v \notin \set{0,c}^\perp} \lambda_{v,H}^2
\]
Under the assumption that $n$ is odd, we have by \Cref{prop:bentcomps-crookedAB-Walsh} that $F|_H$ has $2^{n-1}$ bent components and $2^{n-1}-1$ semi-bent components, implying the sum $\sum_{v \notin \set{0,c}^\perp} \lambda_{v,H}^2$ is equal to 
\begin{align*}
2^{n-1} (|\bentcomps{F|_H} \setminus \set{0,c}^\perp + 4|\semibentcomps{F|_H} \setminus \set{0,c}^\perp|) 
&= 2^{n-1}(2^{n-1} + 3|\semibentcomps{F|_H} \setminus \set{0,c}^\perp|)\\
&= 2^{n-1}(2^{n+1}-3-3|\semibentcomps{F|_H} \cap \set{0,c}^\perp|).
\end{align*}

From the above, it then follows that a crooked function $F$, for $n$ odd, has the strong D-property if and only if for any affine hyperplane $H$ and any nonzero $c$, we have $2^{n-1}(2^{n+1}-3-3|\semibentcomps{F|_H} \cap \set{0,c}^\perp|)<2^{n-1}(3 \cdot 2^{n-1}-2)$, that is, $|\semibentcomps{F|_H} \cap \set{0,c}^\perp| > \frac{2^{n-1}-1}{3}$. 
More precisely:
\begin{proposition}\label{4.17}
Let $n$ be any odd integer, $F$ any crooked $(n,n)$-function and $H$ any affine hyperplane of $\F_2^n$. Then $F|_H$ has the D-property if and only if, for any nonzero $c$, we have $|\semibentcomps{F|_H} \cap \set{0,c}^\perp| > \frac{2^{n-1}-1}{3}$, where $\semibentcomps{F|_H} $ denotes the set of semi-bent components of $F$.
\end{proposition} 

It was shown in \cite{Abbondati2024} that if $n$ is odd and $F$ is strongly plateaued as well as the restriction $F|_H$ of $F$ to a (linear) hyperplane $H$, then $F|_H$ has the D-property if and only if $|\semibentcomps{F|_H} \cap \set{0,c}^\perp| > \frac{2^{n-1}-1}{3}$.
Proposition \ref{4.17} shows that we can obtain the same result on a crooked function $F$ (instead of considering a single restriction of $F$ to an affine hyperplane, we are considering all such restrictions) without assuming its restriction to $H$ is strongly plateaued. 
The fact of thus easing the hypothesis is not benign because we have shown in \cite{CarletThornburghRestrictions} that if a strongly plateaued function is strongly plateaued on every hyperplane, then if $n \geq 4$ it is quadratic (and hence, if $F$ is APN and is strongly plateaued on every affine hyperplane, it is quadratic).
We have then weakened the hypothesis of \cite{Abbondati2024} when considering all restrictions of $F$ to affine hyperplanes.
And it is important to determine which results on quadratic APN functions generalize to crooked functions since the question of whether a non-quadratic crooked function can exist remains open.

\section{Exclude parity adjoints of crooked functions}

In the previous section, we proved a general result that proved \Cref{conj:Gorodilova} as a corollary.
In particular, we defined Boolean functions, which we called \textit{exclude parity functions}, that are associated to any APN function, and we studied them in the particular cases of crooked and quadratic APN functions.
If $n \geq 4$ and $F \colon \F_2^n \to \F_2^n$ is a crooked function, we know that for any nonzero $v \in\F_2^n$, the equality $\deg_{alg}(v \cdot \pi_F)=n-2$ holds if and only if $\deg_{alg}(f_v)=2$, where the exclude parity function $f_v$ is defined by
\[
f_v(u) = (m_u(v) \mod 2) = (\mult_{\graph{F} \cap (\set{0,u}^\perp \times \F_2^n)}(0,F(0)+v) \mod 2),
\]
and \Cref{cor:more-quadratic-cpts-closer-to-conjecture} implies the latter condition always holds when $v \neq 0$ and $F$ is quadratic.
In this section, we prove that when $F$ is crooked, these exclude parity functions are described by the component functions of an $(n,n)$-function.

\begin{definition}  
Assume $n \geq 4$, and let $F \colon \F_2^n \to \F_2^n$ be an APN function.
A vectorial Boolean function $\varepsilon_F \colon \F_2^n \to \F_2^n$ is called the \textit{exclude parity adjoint} of $F$ if
\[
 v \cdot \varepsilon_F(u) = f_v(u) \oplus 1
 \]
for all nonzero $u,v \in \F_2^n$.
\end{definition}

Note that if $n \geq 3$ and $F \colon \F_2^n \to \F_2^n$ is a plateaued APN function with an exclude parity adjoint, then $\varepsilon_F(0)=0$ since $f_v(0) \oplus 1 = 0$ for all nonzero $v \in \F_2^n$ by \Cref{prop:plateaued-oddmults}.

The main result of this section is that every crooked $(n,n)$-function, where $n \geq 4$, has an exclude parity adjoint.
The property of having an exclude parity adjoint is very strong as it is equivalent to the property that for all nonzero $u \in \F_2^n$, the set 
\begin{equation}\label{affine-hyperplane-or-empty}
\set{v \in \F_2^n \setminus \set{0} : \mult_{\graph F \cap (\set{0,u}^\perp \times \F_2^n)}(0,F(0)+v) \text{ is even}}
\end{equation}
is an affine hyperplane or empty.
In particular, we shall show that in the case of $F$ being quadratic APN, the above set is always the complement of a linear hyperplane (which is equivalent to the property that $\varepsilon_F(u)=0$ if and only if $u=0$, and this is also equivalent to $F$ satisfying \Cref{conj:Gorodilova}).

We separate our proof into multiple lemmas. 
First, we characterize when $\varepsilon_F$ exists in terms of a condition on the preimages of $2$-dimensional linear subspaces under $\pi_F$, where $F$ is a crooked $(n,n)$-function and $n \geq 4$.
Then, we will prove the case when $n$ is even by a simple argument.
However, for the case when $n$ is odd, we will require two lemmas relating the existence of $\varepsilon_F$ to the algebraic degree of some Boolean functions associated to the restrictions of $F$ to affine hyperplanes as well as the algebraic degree of $\pi_F^{-1}$.

\begin{lemma}\label{lem:parity-adj-equiv}
    Assume $n \geq 4$, and let $F \colon \F_2^n \to \F_2^n$ be a crooked function. 
    Then $F$ has an exclude parity adjoint if and only if for any 2-dimensional linear subspace $V$ and any nonzero $u \in \F_2^n$, the set
    $\pi_F^{-1}(V^\perp) \setminus \set{0,u}^\perp$ has even size.
\end{lemma}
\begin{proof}
    Let $u \in \F_2^n$ be nonzero, and let $S = \graph{F'}$, where $F' = F|_{\set{0,u}^\perp}$.
    Let $\mathcal{S}_v$ be the set of 2-dimensional linear subspaces $E \subseteq \set{0,u}^\perp$ such that $\sum_{x \in E} F(x) = v$.
    Note that $|\mathcal{S}_v|$ is equal to $m_u(v) = \mult_{\graph{F'}}(0,F(0)+v)$, and so $f_v(u) \equiv |\mathcal{S}_v| \pmod 2$.

    Let $V = \langle v_1, v_2 \rangle$ be an arbitrary 2-dimensional linear subspace of $\F_2^n$, and let $N_V$ be the number of 2-dimensional linear subspaces $E$ contained in $\set{0,u}^\perp$ such that $\sum_{x \in E} F(x) \in V \setminus \set{0}$.
    Let us show that $N_V$ satisfies the relation $N_V+1 \equiv |\pi_F^{-1}(V^\perp) \setminus \set{0,u}^\perp| \pmod 2$.
    First, note that $N_V = m_u(v_1)+m_u(v_2)+m_u(v_1+v_2)$.
    Also, by \Cref{thm:crooked-function-ortho-Walsh}, we have $W_{v \cdot \pi_F}(u)-W_{v \cdot \pi_F}(0) = 8m_u(v)-8m(v)+2^n$ for any $v \in \F_2^n \setminus \set{0}$, implying 
    \[
    m_u(v) + 1 \equiv \frac{W_{v \cdot \pi_F}(u)-W_{v \cdot \pi_F}(0)}{8} \pmod 2.
    \]
    since $m(v)$ is odd by \Cref{prop:plateaued-oddmults} and $n \geq 4$.
    Then $N_V+1 \equiv \frac{1}{8}\sum_{v \in V \setminus \set{0}} (W_{v \cdot \pi_F}(u)-W_{v \cdot \pi_F}(0)) \pmod 2$.
    For any $t \in \F_2^n$, we have 
    \[
    \sum_{v \in V \setminus \set{0}} W_{v \cdot \pi_F}(t) 
    =\sum_{x \in \F_2^n} (-1)^{x \cdot t}\sum_{v \in V \setminus \set{0}} (-1)^{v \cdot \pi_F(x)}
    = \sum_{x \in \F_2^n} (-1)^{x \cdot t} \parens{4 \cdot 1_{V^\perp}(\pi_F(x)) -1}.
    \]
    In particular, we have
    \begin{align*}
        \sum_{v \in V \setminus \set{0}}W_{v \cdot \pi_F}(u) &=  4\sum_{x \in \F_2^n} (-1)^{x \cdot u} 1_{V^\perp}(\pi_F(x)), \\
        \sum_{v \in V \setminus \set{0}} W_{v \cdot \pi_F}(0) &=4 \sum_{x \in \F_2^n} 1_{V^\perp}(\pi_F(x)) - 2^n
    \end{align*}
    with the first equality holding because $u \neq 0$.
    Therefore, 
    \begin{align*}
        \sum_{v \in V \setminus \set{0}} (W_{v \cdot \pi_F}(u)-W_{v \cdot \pi_F}(0))  
        &= 4 \sum_{x \in \F_2^n, \pi_F(x) \in V^\perp}\parens{(-1)^{u \cdot x}-1}+2^n \\
        &= -8|\set{x \in \F_2^n : \pi_F(x) \in V^\perp, u \cdot x = 1}| +2^n.
    \end{align*}
    Since $n \geq 4$, we then have $N_V +1 \equiv |\set{x \in \F_2^n : \pi_F(x) \in V^\perp, u \cdot x = 1}| \pmod 2$.
    In other words, $N_V+1 \equiv |\pi_F^{-1}(V^\perp) \setminus \set{0,u}^\perp| \pmod 2$.
  
    Now, note that $N_V$ is odd if and only if we have $m_u(v_1) + m_u(v_2) + m_u(v_1+v_2) \equiv 1 \pmod 2$, which is the same as $f_{v_1}(u) \oplus f_{v_2}(u) \oplus f_{v_1+v_2}(u)=1$.
    Consequently, $N_V$ is odd for all $2$-dimensional linear subspaces $V$ if and only if $v \mapsto f_v(u) \oplus 1$ is a linear map (for notational convenience, we consider $f_0$ to be identically one).
    This latter condition is then equivalent to the condition that there exists a unique point $e_u \in \F_2^n$ such that $v \cdot e_u = f_v(u) \oplus 1$ for all nonzero $v \in \F_2^n$. 
    That is, there exists a vectorial function $\varepsilon_F \colon \F_2^n \to \F_2^n$ such that $\varepsilon_F(0)=0$ and whose component functions are $v \cdot \varepsilon_F = f_v \oplus 1$.
    The equivalence is then proven.
\end{proof}

We now require the two following lemmas, which we will use in the case of $n\geq 5$ being odd.
\begin{lemma}\label{lem:indbentcpts-upperbound}
    Assume $n$ is even, and let $F \colon \F_2^n \to \F_2^m$ be plateaued.
    Then $\deg_{alg}(1_{\bentcomps{F}}) \leq \min\set{m,\frac{n}{2}}$.
\end{lemma}
\begin{proof}
    By \cite[Proposition 3.1]{beneteau2025walshspectraquadraticapn} (which is stated for quadratic functions, but only uses the properties of plateaued functions), for any subspace $W \subseteq \F_2^n$ of dimension at least $\frac{n}{2}+1$, the intersection $|W \cap (\F_2^n \setminus (\bentcomps{F} \cup \set{0}))|$ is odd.
    Equivalently, $\bentcomps{F}$ intersects any $(\frac{n}{2}+1)$-dimensional linear subspace in an even number of points, and so $\deg_{alg}(1_{\bentcomps{F}}) \leq \frac{n}{2}$.
    The upper bound $\deg_{alg}(1_{\bentcomps{F}}) \leq m$ also holds since $1_{\bentcomps{F}}$ is an $m$-variable Boolean function.
\end{proof}

We now prove the following lemma, which establishes the equivalence of the existence of an exclude parity adjoint function of a crooked function $F \colon \F_2^n \to \F_2^n$, where $n \geq 5$ is odd, with the condition that the algebraic degree of $\pi_F^{-1}$ is bounded above by $n-3$.
Indeed, we will confirm that the latter condition always holds in the proof of \Cref{thm:parityadjoint}.

\begin{lemma}\label{lem:excludeparity-exists-iff-piFinv}
    Assume $n \geq 5$ is odd, and let $F \colon \F_2^n \to \F_2^n$ be a crooked function.
    Then $F$ has an exclude parity adjoint if and only if $\deg_{alg}(\pi_F^{-1}) \leq n-3$.
\end{lemma}
\begin{proof}
    We saw in \Cref{lem:parity-adj-equiv} that $F$ has an exclude parity adjoint if and only if for any 2-dimensional linear subspace $V$ and any nonzero $u \in \F_2^n$, the value of $|\pi_F^{-1}(V^\perp) \setminus \set{0,u}^\perp|$ is even.
    Since $n$ is odd, we know that $\pi_F$ is invertible (see \Cref{cor:crooked-AB-invertibleOrtho}). 
    Hence, 
    \[
    |\pi_F^{-1}(V^\perp) \setminus \set{0,u}^\perp| = \sum_{w \in V^\perp \setminus \set{0}} u \cdot \pi_F^{-1}(w) = \sum_{w \in V^\perp} u \cdot \pi_F^{-1}(w).
    \]
    Therefore, $F$ has an exclude parity adjoint if and only if every component function of $\pi_F^{-1}$ has even weight on any $(n-2)$-dimensional linear subspace.
    Recall that the algebraic degree of a Boolean function is at most $k$ if and only if it has zero-sum on every $(k+1)$-dimensional affine subspace, and indeed it suffices to only consider linear subspaces by Relation (2.4) of \cite{CarletBook}.
    The result immediately follows.
\end{proof}

We now prove the main result of this section.

\begin{theorem}\label{thm:parityadjoint}
    Assume $n \geq 4$, and let $F \colon \F_2^n \to \F_2^n$ be a crooked function.
    Then $F$ has an exclude parity adjoint $\varepsilon_F$.
\end{theorem}
\begin{proof}
It is known that the ortho-derivative of $F$ has the property that for any nonzero $b \in \im(\pi_F)$, the set $\pi_F^{-1}(b) \cup \set{0}$ is a linear subspace of size $2^{-n} \lambda_b^2$, where $\lambda_b$ is the amplitude of $b \cdot F$ \cite{Kyureghyan2007}.
    In particular, if $n$ is even, then the dimension of $\pi_F^{-1}(b) \cup \set{0}$ is even for all $b \in \im(\pi_F)$.
    Since
    \begin{align*}
        N:=|\pi_F^{-1}(V^\perp) \setminus \set{0,u}^\perp|
        &= \sum_{b \in (V^\perp\setminus \set{0}) \cap \im(\pi_F)} |(\pi_F^{-1}(b) \cup \set{0}) \cap (\F_2^n \setminus \set{0,u}^\perp)|, 
    \end{align*}
    we see that $N$ is even when $n$ is even because each term in the sum on the right-hand-side is the intersection of a linear subspace of even (positive) dimension with the affine hyperplane $\F_2^n \setminus \set{0,u}^\perp$, which has even size.
    Therefore, $F$ has an exclude parity adjoint when $n$ is even by \Cref{lem:parity-adj-equiv}.

    Now, assume $n$ is odd, and let $H$ be an affine hyperplane with underlying vector space $\set{0,e}^\perp$.
    From the proof of \Cref{prop:bentcomps-crookedAB-Walsh}, we know that $\bentcomps{F|_H} = \supp(e \cdot \pi_F^{-1})$.
    In other words, 
    \[
    1_{\bentcomps{F|_H}} = e \cdot \pi_F^{-1}.
    \]
    Since $F$ is crooked, its restrictions to affine hyperplanes are plateaued \cite{CarletThornburghRestrictions}, and therefore, we can apply \Cref{lem:indbentcpts-upperbound} to deduce that $\deg_{alg}(e \cdot \pi_F^{-1}) = \deg_{alg}(1_{\bentcomps{F|_H}}) \leq \frac{n-1}{2}$.
    Thus, since $\frac{n-1}{2} \leq n-3$ holds for all $n \geq 5$, we can apply \Cref{lem:excludeparity-exists-iff-piFinv}, implying $F$ has an exclude parity adjoint.
\end{proof}

It is worth noting that if $F$ is a quadratic APN function and $n \geq 4$, then $\varepsilon_F(u)=0$ implies $u=0$.
Otherwise, we would have that $f_v$ is identically one for some nonzero $v \in \F_2^n$, but this is impossible by \Cref{cor:more-quadratic-cpts-closer-to-conjecture}.
In particular, if $F \colon \F_2^n \to \F_2^n$ is a quadratic APN function, then for any $u \neq 0$, the set from (\ref{affine-hyperplane-or-empty}) is always an affine hyperplane.

\begin{remark}
    Let $n \geq 4$, and let $F \colon \F_2^n \to \F_2^n$ be a crooked function.
    It seems natural to attempt to generalize the statement of \Cref{thm:parityadjoint} and prove that if $L \subseteq \F_2^n$ is a linear subspace of codimension at least $1$, then  
    \[
\set{v \in \F_2^n \setminus \set{0} : \mult_{\graph F \cap (L \times \F_2^n)}(0,F(0)+v) \text{ is even}}
    \]
    or 
        \[
\set{v \in \F_2^n \setminus \set{0} : \mult_{\graph F \cap (L \times \F_2^n)}(0,F(0)+v) \text{ is odd}} \cup \set{0}
    \]
    is an affine subspace.
    However, we have verified via computer calculations that if $n=7$ and $F$ is defined on $\F_{2^n}$ by $F(x)=x^3$, then there does not exist a linear subspace $L$ of codimension $2$ satisfying either of the above conditions.
\end{remark}

As previously mentioned, the existence of an exclude parity adjoint of an APN function is a very restrictive condition, and we will see that \Cref{thm:parityadjoint} has several interesting consequences in the following text. 
\Cref{thm:parityadjoint} also leads to the following problem.

\begin{openp}
    As shown in the proof of \Cref{thm:parityadjoint}, we see that when $n\geq5$ is odd and $F$ is a crooked function, we have $\deg_{alg}(\pi_F^{-1}) \leq \frac{n-1}{2}$.
    Is it true that $\deg_{alg}(v \cdot \pi_F^{-1}) = \frac{n-1}{2}$ for all nonzero $v \in \F_2^n$?  
\end{openp}

We also have the following proposition regarding the algebraic degrees of the functions $\pi_{F|_L}^v$ where $F$ is a crooked function and $L$ is a linear hyperplane. 
\begin{proposition}
    Assume $n \geq 4$.
    Let $F \colon \F_2^n \to \F_2^n$ be a crooked function, let $u \in \F_2^n$ be nonzero, and let $L = \set{0,u}^\perp$.
    Let $v \in \F_2^n$ be nonzero.
    Then $\deg_{alg}(\pi_{F|_L}^v) \leq n-3$ if $v \in  \set{0,\varepsilon_F(u)}^\perp$, and $\deg_{alg}(\pi_{F|_L}^v) = n-1$ otherwise.
\end{proposition}
\begin{proof}
    For any nonzero $v$, we know that $m_{L,u}(v) \equiv f_v(u) \equiv v \cdot \varepsilon_F(u) + 1 \pmod{2}$.
    The result then immediately follows from \Cref{prop:crooked-orthoderivative-algdegbound}.
\end{proof}

In the following result, we provide a characterization of when $\varepsilon_F$ is an APN function (indeed, this is possible as we will see in \Cref{subsec:Gold}), where $F \colon \F_2^n \to \F_2^n$ is a crooked function and $n \geq 4$.
This characterization involves the notion of $k$th order sum-freeness, where we say that a function $G \colon \F_2^n \to \F_2^n$ is \textit{$k$th order sum-free} if for any $k$-dimensional affine space $A \subseteq \F_2^n$, we have $\sum_{x \in A} G(x) \neq 0$ \cite{CarletSumFree2025}.
Note that if $\deg_{alg}(G)\leq k$, then $G$ is $k$th order sum-free if and only if $\sum_{x \in L}G(x) \neq 0$ for all $k$-dimensional linear subspaces $L \subseteq \F_2^n$ (of course, if $\deg_{alg}(G) < k$, then this cannot happen since the sum of values that $G$ takes on any $k$-dimensional affine subspace would then be zero).  
Indeed, if $\deg_{alg}(G)\leq k$ and $L = \langle e_1, \dots, e_k\rangle$, then $\sum_{x \in t+L} G(x) = D_{e_1} \cdots D_{e_k}G(t)$ does not depend on the value of $t \in \F_2^n$.

\begin{theorem}\label{prop:adjoint-APN-iff-sumfree}
    Assume $n \geq 4$, and let $F \colon \F_2^n \to \F_2^n$ be a crooked function.
    Then $\varepsilon_F$ is APN if and only if $\pi_F$ is $(n-2)$th order sum-free.
\end{theorem}
\begin{proof}
    From the proof of \Cref{lem:parity-adj-equiv}, for any nonzero $u,v \in \F_2^n$, we know that $v \cdot \varepsilon_F (u) \equiv  \frac{1}{8}\parens{W_{v \cdot \pi_F}(u) - W_{v \cdot \pi_F}(0)} \pmod 2$.
    Fix $v \neq 0$, and let $u_1, u_2 \in \F_2^n$ be linearly independent, and let $U = \langle u_1, u_2 \rangle$.
    The Poisson summation formula \cite[Relation 2.39]{CarletBook} provides the equality 
    \[
    \sum_{u \in U}W_{v \cdot \pi_F}(u)=4\sum_{x \in U^\perp}(-1)^{v \cdot \pi_F(x)}.
    \]
    Then
    $v\cdot \beta_{\varepsilon_F}(u_1, u_2)=
    v \cdot (\varepsilon_F(u_1)+\varepsilon_F(u_2) + \varepsilon_F(u_1+u_2))$
    is congruent to 
    \begin{align*}
    &\frac{W_{v \cdot \pi_F}(u_1)+W_{v \cdot \pi_F}(u_2)+W_{v \cdot \pi_F}(u_1+u_2)-3W_{v \cdot \pi_F}(0)}{8}\\
    &= \frac{1}{2}\parens{ \sum_{x \in U^\perp} (-1)^{v \cdot \pi_F(x)} - W_{v \cdot \pi_F}(0)} \\
    &= \frac{1}{2}\parens{2^{n-2} - 2\wt(v \cdot \pi_F|_{U^\perp}) - (2^n-2\wt(v \cdot \pi_F)}\\
    &= 2^{n-3} - \wt(v \cdot \pi_F|_{U^\perp}) - 2^{n-1} + \wt(v \cdot \pi_F).
    \end{align*}
    modulo $2$.
    Since $\deg_{alg}(v \cdot \pi_F) \leq n-2 < n$ by \Cref{prop:crooked-orthoderivative-algdegbound}, we know $\wt(v \cdot \pi_F)$ is even.
    Combining this fact with the inequality $n \geq 4$ provides 
    $v \cdot \beta_{\varepsilon_F}(u_1, u_2)\equiv \wt(v \cdot \pi_F|_{U^\perp}) \pmod{2}$.
    Equivalently, $v \cdot \beta_{\varepsilon_F}(u_1, u_2)= v \cdot \sum_{x \in U^\perp} \pi_F(x)$, and since $v$ was arbitrary this implies that $\beta_{\varepsilon_F}(u_1, u_2) = \sum_{x \in U^\perp}\pi_F(x)$.
    Since $\varepsilon_F$ is quadratic, it is APN if and only if $\beta_{\varepsilon_F}(a,b)=0$ implies $a,b$ are linearly dependent.
    In other words, $\varepsilon_F$ is APN if and only if the sum of $\pi_F$ over any $(n-2)$-dimensional linear subspace is nonzero, i.e. $\pi_F$ is $(n-2)$th order sum-free.
\end{proof}

Let us provide an equivalent description of when $\varepsilon_F$ is APN, for a quadratic APN function $F$ and $n \geq4$.
For an $(n,m)$-function $G$ and an affine subspace $A \subseteq \F_2^n$, we say that $A$ is a \textit{degree-drop subspace} of $G$ if $\deg_{alg}(G|_A)<\deg_{alg}(G)$ (in \cite{CarletDegreeStability2025}, the original definition was for Boolean functions only, but we generalize it to vectorial functions).
The largest $k$ such that $G$ does not have a degree-drop subspace of codimension $k$ is called the \textit{restriction degree stability} $\deg_{stab}(G)$ of $G$ \cite{CarletDegreeStability2025}.
Since for a quadratic APN function $F$ and $n \geq 4$, we know that $\deg_{alg}(\pi_F) = n-2$, it follows that $\varepsilon_F$ is APN if and only if the restriction degree stability of $\pi_F$ is $2$, that is, for any linear subspace $L \subseteq \F_2^n$ of codimension $2$, there exists a nonzero $v \in \F_2^n$ such that $v\cdot\pi_F$ does not have $L$ as a degree-drop subspace.

\Cref{prop:adjoint-APN-iff-sumfree} then provides a condition to be potentially utilized in the search of new APN functions. 
\begin{openp}
Find crooked functions whose ortho-derivatives are $(n-2)$th order sum-free.
\end{openp}

We also prove that if two crooked functions are EA equivalent, then their exclude parity adjoints are linearly equivalent.

\begin{proposition}
    Assume $n \geq 4$.
    Let $F,G \colon \F_2^n \to \F_2^n$ be crooked functions.
    If $F$ and $G$ are EA equivalent, then $\varepsilon_F$ and $\varepsilon_G$ are linearly equivalent.
\end{proposition}
\begin{proof}
    Suppose there exist affine permutations $A_1, A_2 \colon \F_2^n \to \F_2^n$ and an arbitrary affine map $A_3 \colon \F_2^n \to \F_2^n$ such that $G = A_2 \circ F \circ A_1 + A_3$.
    For $1 \leq i \leq 3$, let $A_i = L_i + a_i$, where $L_i \colon \F_2^n \to \F_2^n$ is a linear map and $a_i \in \F_2^n$.
    
    It was shown in \cite[Proposition 36]{CanteautRecovering} that when $F$ is quadratic APN, we have $\pi_G = (L_2^{-1})^\ast \circ \pi_F \circ L_1$, where $(L_2^{-1})^\ast$ denotes the adjoint operator of $L_2^{-1}$.
    However, this straightforwardly generalizes to crooked functions too.
    Indeed, we have $D_a G(x) = L_2(D_{L_1(a)} F(A_1(x))) + L_3(a)$ and $\im(D_a G)$ then equals $L_2( \im(D_{L_1(a)} F)) +L_3(a)$, and therefore, the underlying linear hyperplane $H_{G,a}$ of $\im(D_a G)$ equals $L_2(H_{F,L_1(a)})$ where $H_{F,L_1(a)}$ is the underlying linear hyperplane of $\im(D_{L_1(a)}F)$.
    Since $\pi_G(a)$ is characterized by the fact it is orthogonal to $H_{G,a} = L_2(H_{F,L_1(a)})$, we have $L_2^\ast (\pi_G(a)) = \pi_F(L_1(a))$, and so $\pi_G = (L_2^{-1})^\ast \circ \pi_F \circ L_1$.
    
    From the proof of \Cref{lem:parity-adj-equiv}, for any nonzero $u,v \in \F_2^n$, we have the relation $v \cdot \varepsilon_F (u) \equiv  \frac{1}{8}\parens{W_{v \cdot \pi_F}(u) - W_{v \cdot \pi_F}(0)} \pmod 2$ and $v \cdot \varepsilon_G (u) \equiv  \frac{1}{8}\parens{W_{v \cdot \pi_G}(u) - W_{v \cdot \pi_G}(0)} \pmod 2$.
    For any nonzero $v \in \F_2^n$ and $x \in \F_2^n$, we have $v \cdot \pi_G(x) = v \cdot (L_2^{-1})^\ast(\pi_F(L_1(x))) = L_2^{-1}(v) \cdot \pi_F(L_1(x))$, and so 
    \[
    W_{v \cdot \pi_G}(u) = \sum_{x \in \F_2^n}(-1)^{L_2^{-1}(v) \cdot \pi_F(L_1(x)) + u \cdot x} = W_{L_2^{-1}(v) \cdot \pi_F}((L_1^{-1})^\ast(u))
    \]
    Thus, for any nonzero $u,v \in \F_2^n$, we have 
    $v \cdot \varepsilon_G(u) = L_2^{-1}(v) \cdot \varepsilon_F((L_1^{-1})^\ast(u))$, implying 
    \[
    \varepsilon_G(u) = ( (L_2^{-1})^\ast \circ \varepsilon_F \circ (L_1^{-1})^\ast)(u)
    \]
    for all $u \in \F_2^n$ (with equality at $0$ holding by the definition of a parity adjoint function).
\end{proof} 

Thus, for $n \geq 4$, the exclude parity adjoint of a crooked function $F\colon \F_2^n \to \F_2^n$ can be used when testing EA equivalence of $F$ to another function $G\colon \F_2^n \to \F_2^n$. 
However, our above methods of determining $\varepsilon_F$ all rely on first knowing the ortho-derivative of $F$, which is regarded as highly discriminating with respect to the EA equivalence class of $F$.
It remains an open problem to determine whether or not one can compute $\varepsilon_F$ quicker than the best-known algorithms to compute $\pi_F$, for a given crooked function $F$.

\subsection{Finding an explicit description of $\varepsilon_F(u)$}
So, far, we have not provided an explicit form for the exclude parity adjoint $\varepsilon_F$ of a crooked $(n,n)$-function, where $n \geq 4$ (besides of course, describing its component functions).
However, we will now provide an equality, in the case of $n\geq4$ even, that exactly describes $\varepsilon_F$ as a vectorial function, rather than only via its component functions.
Recall that we write $\semibentcomps{F} = \set{b \in \F_2^n : b \cdot F \text{ is semi-bent}}$.

\begin{proposition}\label{prop:n-even-excludeParityAdjExplicit}
    Assume $n \geq 4$, and let $F$ be a crooked $(n,n)$-function.
    If $n$ is even, then for any $u \in \F_2^n$, we have
    \[
    \varepsilon_F(u)=\sum_{b \in X_u}b,
    \]
    where $X_u = \set{b \in \semibentcomps{F} : u \notin \mathcal{E}_{b \cdot F}^\perp}$ and $|X_u| \equiv 0 \pmod 4$.
    If $n$ is odd, then for any $u,v \in \F_2^n$ where $v \neq 0$, we have 
    \[
    v \cdot \varepsilon_F(u) \equiv \frac{1}{2} \sum_{b \in Y_u} v \cdot b 
    \pmod{2},
    \]
    where $Y_u = \set{b \in \F_2^n \setminus \set{0} : u \notin \mathcal{E}_{b \cdot F}^\perp}$ and $|Y_u| \equiv 0 \pmod 8$.
\end{proposition}
\begin{proof}
    Let $u,v \neq 0$, and let $H = \set{0,u}^\perp$. 
    From \Cref{rem:plateaued-restriction-mults}, we know that 
    \[
    m(v) = \frac{1}{6 \cdot 2^n}\sum_{b \in \F_2^n} (-1)^{v \cdot b} \lambda_b^2
    \quad\text{and} \quad
    m_u(v) = \frac{1}{6\cdot 2^n}\sum_{b \in \F_2^n}(-1)^{v \cdot b}\lambda_{b,H}^2, 
    \]
     where $\lambda_b$ (resp. $\lambda_{b,H}$) is the amplitude of $b \cdot F$ (resp. $b \cdot F|_H$). 
    Recall that we know (also see \cite[Section 4]{CarletThornburghRestrictions}) that $\lambda_{b,H}$ is equal to $\lambda_b$ if $u \in \mathcal{E}_{b \cdot F}^\perp$, and $\lambda_{b,H} = \frac{1}{2}\lambda_b$ otherwise.
    Therefore, 
    \[
    m(v) - m_u(v) = \frac{1}{2^{n+3}} \sum_{\substack{b \in \F_2^n \\ u \notin \mathcal{E}_{b \cdot F}^\perp}}(-1)^{v \cdot b} \lambda_b^2.
    \]
    We will now reduce both sides of the above equation modulo $2$, but we first make some observations. 
    First, note that if $b \cdot F$ is not bent nor semi-bent, then $\lambda_b^2 \geq 2^{n+4}$, implying $\frac{\lambda_b^2}{2^{n+3}}$ is even.
    If $b \cdot F$ is bent, then all derivatives of $b \cdot F$ are balanced and cannot be constant, implying $\mathcal{E}_{b \cdot F}^\perp  = \F_2^n$.
    Upon combining these observations and also applying \Cref{prop:plateaued-oddmults} (which implies $m(v) \equiv 1 \mod 2$), we have 
    \[
    v \cdot \varepsilon_F(u) \equiv m_u(v) +1 \equiv  \frac{1}{2^{n+3}} \sum_{\substack{b \in \semibentcomps{F} \\ u \notin \mathcal{E}_{b \cdot F}^\perp}} (-1)^{v \cdot b} \lambda_b^2
    \equiv \frac{1}{2} \sum_{\substack{b \in \semibentcomps{F} \\ u \notin \mathcal{E}_{b \cdot F}^\perp}} (-1)^{v \cdot b}
    \pmod{2}.
    \]
    Let $X_u = \set{b \in \semibentcomps{F} : u \notin \mathcal{E}_{b \cdot F}^\perp}$.
    Then 
    \[
    v \cdot \varepsilon_F(u) \equiv \frac{|X_u|}{2} - \sum_{b \in X_u} v \cdot b
    \equiv \frac{|X_u|}{2} + v \cdot \sum_{b \in X_u} b \pmod{2}.
    \]
    Considering any nonzero $v'\neq v$, we have $(v+v') \cdot \sum_{b \in X_u} b \equiv \frac{|X_u|}{2} +v \cdot \sum_{b \in X_u} b +  v'\cdot \sum_{b \in X_u} b \pmod{2}$, implying $X_u$ has size divisible by four.
    Therefore, for any $u,v \in \F_2^n$ (the cases where $u$ or $v$ is are immediate), we have $v \cdot \varepsilon_F(u) = v \cdot \sum_{b \in X_u} b$, and it immediately follows that $\varepsilon_F(u) = \sum_{b \in X_u} b$, as desired.

    Now, consider when $n$ is odd. 
    Then, all component functions of $F$ have amplitude $2^{\frac{n+1}{2}}$ since $F$ is AB, implying (using the same reasoning as above) that for nonzero $u,v \in \F_2^n$, we have
    \[
    v \cdot \varepsilon_F(u) \equiv \frac{1}{4} \sum_{\substack{b \in \F_2^n \setminus \set{0} \\ u \notin \mathcal{E}_{b \cdot F}^\perp}} (-1)^{v \cdot b} 
    \equiv \frac{|Y_u|}{4} + \frac{1}{2} \sum_{b \in Y_u} v \cdot b
    \pmod{2},
    \]
    where $Y_u = \set{b \in \F_2^n \setminus \set{0} : u \notin \mathcal{E}_{b \cdot F}^\perp}$.
    Similar to before, it is straightforward that $|Y_u|$ is divisible by $8$, and the result follows.
\end{proof}

We then have the following lower bound on the number of semi-bent components of a crooked function over $\F_2^n$ where $n \geq 4$ is even.

\begin{corollary}\label{cor:n-semi-bent}
    Assume $n \geq 4$ is even, and let $F$ be a quadratic APN $(n,n)$-function.
    Then $F$ has at least $n$ semi-bent components.
\end{corollary}
\begin{proof}
    Since $F$ is quadratic APN, its exclude parity adjoint $\varepsilon_F$ has component functions all of algebraic degree $2$.
    In particular, no component function of $\varepsilon_F$ is constant implying its image cannot be contained in a hyperplane.
    Moreover, we know from \Cref{prop:n-even-excludeParityAdjExplicit} that the span of $\semibentcomps{F}$ contains $\im(\varepsilon_F)$, and so $\semibentcomps{F}$ spans $\F_2^n$, providing $|\semibentcomps{F}|\geq n$.
\end{proof}

\begin{remark}
    One can show that if $n \geq 4$ is even $F\colon \F_2^n \to \F_2^n$ is a crooked function such that $\deg_{alg}(\pi_F)=n-2$, then $F$ has at least $5$ semi-bent components.
    Indeed, if $\pi_F$ has a component of algebraic degree $n-2$ (e.g. in the case that $F$ has at least one quadratic component, see \Cref{cor:more-quadratic-cpts-closer-to-conjecture}), then by \Cref{prop:orthoderiv-algdeg-equiv} there exist some linearly independent $e_1, e_2 \in \F_2^n$ satisfying the condition from \Cref{prop:algdeg-characterization-semibent}, and this immediately implies $\semibentcomps{F} \neq \emptyset$.
    This non-emptiness result is necessary to apply a result of Heden \cite{HedenLengthTail2009}.
    It is known that for any nonzero $b$ in the image of $\pi_F$, the set $\pi_F^{-1}(b) \cup \set{0}$ is a linear subspace of size $2^{-n} \lambda_b^2$, where $\lambda_b$ is the amplitude of $b \cdot F$  \cite{Kyureghyan2007}, and note that this implies $\pi_F^{-1}(b) \cup \set{0}$ has even dimension for $b \in \im(\pi_F)$ (of course, we then have $b \in \im(\pi_F)$ is equivalent to $b \cdot F$ is non-bent).
    The collection $\set{\set{0} \cup \pi_F^{-1}(b) : b \in \im(\pi_F}$ is then a collection of linear subspaces that trivially intersect and such that their union is $\F_2^n$, and such a collection is called a \textit{vector space partition}.
    Heden's result provides a lower bound the number of the number of subspaces of smallest dimension in any vector space partition, and by a straightforward argument it follows $|\semibentcomps{F}|\geq 5$.
    However, we omit the details as this bound is only stronger than \Cref{cor:n-semi-bent} in the case of $n=4$.
\end{remark}

\begin{remark}
    For a plateaued APN function $F \colon  \F_2^n \to \F_2^n$, we say that the amplitude distribution of $F$ is $[i_1^{n_1}, \dots, i_k^{n_k}]$ if $\sum_{j=1}^k n_j=2^n-1$ and $i_1<\cdots <i_k$ and $F$ has exactly $n_j$ components of amplitude $2^{(n+i_j)/2}$.
    For even $n \leq 8$, it was already known that any quadratic APN function has at least one semi-bent component because their exact amplitude distributions are known, see \cite{beneteau2025walshspectraquadraticapn}. 
    It was proven in \cite[Theorem 4.6]{beneteau2025walshspectraquadraticapn} that any quadratic APN function $F \colon \F_2^{10} \to \F_2^{10}$ has an amplitude distribution of the form: $[0^{766},2^{256}, 8^1]$, or $[0^{702+4i}, 2^{320-5i},4^{i}, 6^1]$ with $0 \leq i \leq 64$, or $[0^{628+4i}, 2^{341-5i}, 4^{i}]$ with $0 \leq i \leq 65$.
    Out of these $133$ possible amplitude distributions, $[0^{958}, 4^{64}, 6^1]$ and $[0^{954}, 2^5, 4^{63}, 6^1]$ are two amplitude distributions that have less than 10 semi-bent components, and upon applying \Cref{cor:n-semi-bent}, we then know that neither of these are possible.
\end{remark}

\begin{openp}
    Improve the lower bounds on the number of semi-bent components of a quadratic APN function over $\F_2^n$, where $n$ is even.
\end{openp}

\subsection{Gold APN functions}\label{subsec:Gold}

We have verified computationally that the exclude parity adjoint of a quadratic APN function $F$ is usually different from $F$ and need not be CCZ equivalent to $F$.
However, it is possible for the exclude parity adjoint of a quadratic APN function to be APN itself, and in the case of a Gold APN function $F$, we have $\varepsilon_F = F$.

\begin{proposition}\label{prop:Gold-adj}
    Assume $n \geq 4$, and let $F \colon \F_{2^n} \to \F_{2^n}$ be the Gold APN function defined by $F(x)=x^d$, where $d = 2^i +1$ and $\gcd(i,n)=1$.
    Then $\varepsilon_F = F$.
\end{proposition}
\begin{proof}
    Throughout this proof, we will always consider $u,v,a$ to be nonzero elements of $\F_{2^n}$.
    Recall from \Cref{sec:algdeg} that $m_u(v)$ is equal to the number of 2-dimensional linear subspaces $E \subseteq \{0,u\}^\perp$ such that $\sum_{x \in E}F(x)=v$.
    Hence, $m_u(v) = m_{ua^{-1}}(a^dv)$ for any nonzero $a$ because if $E \subseteq \set{0,u}^\perp$ is a 2-dimensional subspace such that $\sum_{x \in E} x^d = v$, then $\sum_{y \in aE} y^d = a^d \sum_{x \in E} x^d$ and $aE \subseteq \{0,a^{-1}u\}^\perp$.
    Since $\Tr(v\varepsilon_F(u)) \equiv m_u(v) + 1 \pmod 2$, we have $\Tr(a^d v \varepsilon_F(a^{-1}u)) = \Tr(v \varepsilon_F(u))$ for any nonzero $a$.
    This implies $\varepsilon_F(u) = a^d \varepsilon_F(a^{-1}u)$ because $u,v,a$ are arbitrary.
    Taking $c = \varepsilon_F(1)$, we then know that $\varepsilon_F(x) = c x^d$ for all $x \in \F_{2^n}$.

    Also, note that $m_u(v) = m_{u^2}(v^2)$ because $E \subseteq \set{0,u}^\perp$ is a 2-dimensional subspace such that $\sum_{x \in E} x^d = v$, then $\sum_{x \in E}(x^2)^d = v^2$ and $\set{x^2 : x \in E} \subseteq \{0,u^2\}^\perp$ with the latter condition holding because $\Tr(x^2u^2) =\Tr(xu)$ for all $x \in \F_{2^n}$.
    Therefore, $\Tr(v\varepsilon_F(u)) = \Tr(v^2 \varepsilon_F(u^2))$, and we have 
    $\Tr(cvu^d) = \Tr(cv^2 u^{2d}) = \Tr((cv^2u^{2d})^{2^{n-1}}) = \Tr(c^{2^{n-1}} v u^d)$.
    Since $\Tr((c+c^{2^{n-1}})vu^d) = 0$ holds for all nonzero $u,v$, we have $c = c^{2^{n-1}}$, i.e. $c \in \F_2$.
    Moreover, $c = 1$ because $\varepsilon_F(1)=c$ and $\varepsilon_F$ only takes value $0$ at $0$ since $F$ is quadratic.
    Therefore, $\varepsilon_F(x) = x^d=F(x)$ for all $x\in \F_{2^n}$.
\end{proof}

The above proposition tells us that $\varepsilon_F=F$ when $F$ is a Gold APN function (and $n \geq 4$), and so we say that $F$ is \textit{exclude parity self-adjoint}.
This leads to the following interesting problem.

\begin{openp}
Find new examples of quadratic APN functions that are exclude parity self-adjoint.
\end{openp}

It is also worth noting that \Cref{prop:Gold-adj} provides an infinite family of exclude parity adjoints of quadratic APN functions that are all permutations.
For $n \geq 4$ and a quadratic APN function $F$, the condition that $\varepsilon_F$ is a permutation is (by definition) equivalent to the condition that for any $v \neq 0$, the exclude parity functions $f_v(u) = (m_u(v) \mod 2)$, where $v \neq 0$, are balanced.

Gold functions being exclude parity self-adjoint also has additional consequences on their ortho-derivatives due to \Cref{thm:crooked-function-ortho-Walsh}.
Let $F \colon \F_{2^n} \to \F_{2^n}$ be the Gold APN function $F(x) = x^{2^i+1}$, where $\gcd(i,n)=1$.
As recalled in \Cref{prop:Gold-adj}, we know $\Tr(v \pi_F(u)) = \Tr(\frac{v}{u^{2^i+1}})$ for all nonzero $u,v \in \F_{2^n}$, and this implies that $\pi_F(x) = x^{-(2^i+1)}$ for all $x \in \F_{2^n}$ (we still write $0^{-1} := 0$). 
However, note that this provides 
\[
    W_{\pi_F}(u,v) = \sum_{x \in \F_{2^n}}(-1)^{\Tr(ux + vx^{-(2^i+1)})}
    = 
     \sum_{x \in \F_{2^n}} (-1)^{\Tr(vx^{2^i+1} + ux^{-1})}, 
\]
via the change of variables $x \mapsto x^{-1}$, and this latter summation is closely related to the exponential sum
\[
    G^{(i)}_n = \sum_{x \in \F_{2^n}^\ast} (-1)^{\Tr(x^{2^i+1} + x^{-1})}
\]
from \cite{JohansenHellesethKholoshamseq}.
The following conjecture from 2009 is still an open problem.

\begin{conjecture}[\cite{JohansenHellesethKholoshamseq}]\label{conj:mseq-general}
    For any positive integer $i$, 
    $G^{(i)}_n = G^{(\gcd(i,n))}_n$,
    i.e. $G^{(i)}_n$ only depends on $\gcd(i,n)$.
    In particular, when $\gcd(i,n)=1$ then 
    \[
    G^{(i)}_n = \sum_{x \in \F_{2^n}^\ast}(-1)^{\Tr(x^3 + x^{-1})}.
    \]
\end{conjecture}
In \cite{JohansenHellesethKholoshamseq}, it was shown that \Cref{conj:mseq-general} holds when $n$ is odd and $i=2$,
and it was also shown \cite[Proposition 2]{LiuHarrisonLuo} that \Cref{conj:mseq-general} holds when $i=3$ and $n \not\equiv 0 \pmod 3$ (without the assumption that $n$ is odd).
However, \Cref{conj:mseq-general} remains open in its full generality.

We derive a partial result on the above problem as a corollary of \Cref{thm:crooked-function-ortho-Walsh} and \Cref{prop:Gold-adj}.
In particular, we prove that 
\[
G^{(i)}_n \equiv G^{(1)}_n \pmod{16}
\]
for all $n \geq 4$ when $\gcd(i,n)=1$. 
We do this by first determining the congruence class of $W_{\pi_F}(u,v)$ modulo $16$ when $F$ is a Gold APN function.

\begin{proposition}\label{prop:Gold-expsum-congruence}
    Let $n \geq 4$, and let $F \colon \F_{2^n} \to \F_{2^n}$ be the Gold APN function $F(x) =x^{2^i+1}$ where $\gcd(i,n)=1$.
    For any $u,v \in \F_{2^n}^\ast$, we have  
    \[
   W_{\pi_F}(u,v)
    \equiv
     \begin{cases}
         8 \Tr(vu^{2^i+1}) & \text{ if $n$ is odd}, \\ 
        8(\Tr(vu^{2^i+1})+1) & \text{ if $n=4$ and $v$ is a cube}, \\ 
        4+8\Tr(vu^{2^i+1})  & \text{ if $n=4$ and $v$ is not a cube}, \\ 
        8\Tr(vu^{2^i+1})  & \text{ if $n=6$ and $v$ is a cube}, \\ 
        8(\Tr(vu^{2^i+1})+1)  & \text{ if $n=6$ and $v$ is not a cube}, \\ 
        8 \Tr(vu^{2^i+1}) & \text{ if $n \geq 8$ is even},
     \end{cases}
     \pmod{16}.
    \]
\end{proposition}
\begin{proof}
    Fix $u,v \in \F_{2^n}^\ast$, and let $F \colon \F_{2^n} \to \F_{2^n}$ be the Gold APN function $F(x) =x^{2^i+1}$.
    By \Cref{thm:crooked-function-ortho-Walsh}, we have $W_{\pi_F}(u,v) = 8m_u(v) -2m(v)+2$.  
    Also, $m_u(v) \equiv 1 \pmod 2$ if and only if $\Tr(vu^{2^i+1})=\Tr(v \varepsilon_F(u))=0$ by \Cref{prop:Gold-adj}.
    If $n$ is odd, then $m(v) = \frac{2^{n-1}-1}{3}$ because $F$ is AB (see \cite{vandamflass}), and so 
    \[
    W_{\pi_F}(u,v) = 8m_u(v) - \frac{2^n-2}{3} +2 \equiv 8(\Tr(vu^{2^i+1})+1) + 8 \equiv 8 \Tr(vu^{2^i+1}) \pmod{16}.
    \]
    Now, assume $n$ is even.
    From \cite[Corollary 4.18]{MihailaThornburgh2026}, we know 
    \[
    m(v) = 
    \begin{cases}
         \frac{2^n + (-2)^{\frac{n}{2}+1} -2 }{6} & v \text{ is a cube}, \\
         \frac{2^n + (-2)^{\frac{n}{2}}-2}{6} & v \text{ is not a cube}.
    \end{cases}
    \]
    Let us determine $-2m(v) + 2 \pmod{16}$ in both cases.
    Note that $3^{-1} \equiv 11 \pmod{16}$.
    So, if $v$ is a cube, then 
    \[
    -2m(v) + 2 = \frac{8-2^n - (-2)^{\frac{n}{2}+1}}{3} 
    \equiv 11(8-(-2)^{\frac{n}{2}+1}) \pmod{16},
    \]
    which is clearly congruent to $0$ if $n =4$ and $8$ if $n \geq 6$.
    On the other hand, if $v$ is not a cube, then 
    \[
    -2m(v) + 2 = \frac{8-2^n -(-2)^{\frac{n}{2}}}{3} \equiv 11(8-(-2)^{\frac{n}{2}}) \pmod{16},
    \]
    which is congruent to $12, 0,$ or $8$ if $n=4$, $n=6$, or $n \geq 8$, respectively.
    The result follows from the above cases and the congruence $W_{\pi_F}(u,v) = 8(\Tr(vu^{2^i+1}) + 1)  -2m(v) + 2 \pmod{16}$.
\end{proof}

We now prove the claimed result on $G^{(i)}_n$ when $\gcd(i,n)=1$.

\begin{corollary}
    Let $i$ be a positive integer such that $\gcd(i,n)=1$.
    When $n \geq 4$, for any nonzero $u,v \in \F_{2^n}^\ast$, we have
    \[
    \sum_{x \in \F_{2^n}^\ast} (-1)^{\Tr(vx^{2^i+1} + ux^{-1})} \equiv 
    \sum_{x \in \F_{2^n}^\ast} (-1)^{\Tr(vx^3 + ux^{-1})} \pmod{16}
    \]
    if and only if $\Tr(v(u^{2^i+1}+u^3))=0$.
    In particular, $G^{(i)}_n \equiv G^{(1)}_n \pmod{16}$ for all $n$.
\end{corollary}
\begin{proof}
    Taking $F_i(x) = x^{2^i+1}$ and $F(x) = x^3$, note that 
    \begin{align*}
        W_{\pi_{F_i}}(u,v) &= \sum_{x \in \F_{2^n}^\ast} (-1)^{\Tr(vx^{2^i+1} + ux^{-1})} +1, \\
        W_{\pi_F}(u,v) &= \sum_{x \in \F_{2^n}^\ast} (-1)^{\Tr(vx^3 + ux^{-1})}+1,
    \end{align*}
    for any $u,v \in \F_{2^n}^\ast$.
    The first statement immediately follows from \Cref{prop:Gold-expsum-congruence}.
    Meanwhile, this provides $G^{(i)}_n \equiv G^{(1)}_n \pmod{16}$ by taking $u=v=1$ (with the cases of $n \leq 3$ being easily verified by computer calculations).
\end{proof}

We can also characterize when we have equality between $G^{(i)}_n$ and $G^{(1)}_n$, when $\gcd(i,n)=1$, in the following combinatorial manner.

\begin{proposition}\label{prop:equiv-problem-mseq}
    Let $F,F_i \colon \F_{2^n} \to \F_{2^n}$ be the Gold APN functions defined by $F(x)=x^3$ and $F_i(x)=x^{2^i+1}$, where $\gcd(i,n)=1$.
    Then $G^{(i)}_n = G^{(1)}_n$ if and only if 
    $\mult_{\graph{F|_H}}(0,1) = \mult_{\graph{F_i|_H}}(0,1)$ where $H = \set{x \in \F_{2^n} : \Tr(x)=0}$, or equivalently, 
    \[
    |\{(a,b) \in H^2 : a^{2^i}b + ab^{2^i}=1\}| = |\{(a,b) \in H^2 : a^2b + ab^2=1\}|.
    \]
\end{proposition}
\begin{proof}
    By similar reasoning as seen in \Cref{prop:Gold-expsum-congruence}, we have
    \begin{align*}
        G^{(i)}_n &= W_{\pi_{F_i}}(1,1)-1 = 8\mult_{\graph{F_i|_H}}(0,1) -2\mult_{\graph{F_i}}(0,1) + 2, \\
        G^{(1)}_n &= W_{\pi_F}(1,1)-1 = 8\mult_{\graph{F|_H}}(0,1) -2\mult_{\graph{F}}(0,1) + 2.
    \end{align*}
    The first claim follows by noting that $\mult_{\graph{F_i}}(0,1) = \mult_{\graph{F}}(0,1)$ (this is easily seen from the results of \cite{vandamflass,MihailaThornburgh2026} that we recalled in \Cref{prop:Gold-expsum-congruence}).
    Furthermore, the value of $\mult_{\graph{F_i|_H}}(0,1)$ is equal to the number of 2-dimensional subspaces $E$ contained in $H$ such that $\sum_{x \in E} x^{2^i+1} = 1$.    
    However, note that if $E=\langle a,b \rangle$, then $\sum_{x \in E} x^{2^i+1} = a^{2^i+1} + b^{2^i+1} + (a+b)^{2^i+1} = a^{2^i}b + ab^{2^i}$, and so the final equivalence holds too.
\end{proof}

\subsection{The Kim APN function}
It is a well-known result of \cite{browningMcQuisttanWolfe} that the Kim function $\kappa \colon \F_{2^6} \to \F_{2^6}$ defined by $\kappa(x) = x^3 + x^{10} + \alpha x^{24}$, where $\alpha \in \F_{2^6}$ is a primitive element, is CCZ equivalent to a permutation, and the CCZ equivalence class of $\kappa$ contains the only known examples of APN permutations in an even number of variables.
In this subsection, we note a remarkable property of the exclude parity adjoint of $\kappa$.

A \textit{partial difference set} $D \subseteq G$ of an additive group $G$ has the property that every non-identity element of $D$ (resp. $G \setminus D$) can be written as $x-y$ with distinct elements $x,y \in D$ exactly $\lambda$ (resp. $\mu$) ways, and we say that the parameters of $D$ are $(v,k,\lambda, \mu)$ where $v=|G|$ and $k=|D|$.
Many examples of plateaued APN functions have PDSs as their images.
Such functions include plateaued APN functions that are 3-to-1 on $\F_2^n$ minus a single point \cite[Theorem 4]{Budaghyan2023} (note that \cite[Theorem 4]{Budaghyan2023} is written for crooked 3-to-1 functions but clearly generalizes to plateaued 3-to-1 functions).

It also is well-known that $\kappa$ has a PDS as its image, and it has parameters $(64,36,20,20)$.
However, interestingly, $\im(\pi_\kappa)$ is a PDS with parameters $(64, 22, 10,6)$ and $\im(\varepsilon_\kappa)$ is a PDS with the same parameters as $\im(\kappa)$.
Furthermore, there exists a linear automorphism of $\F_{2^6}$ mapping $\im(\kappa)$ to $\im(\varepsilon_\kappa)$.
Indeed, this is due to a nontrivial relation as $\varepsilon_\kappa$ is differentially 4-uniform and not APN.

\section{On the algebraic degrees of Boolean functions associated to the components of plateaued APN functions and their restrictions}\label{sec:fcns-assoc-to-cpts}

In this section, we consider various indicator functions related to the component functions of plateaued APN functions (with the primary focus being on quadratic APN functions).
For any function $F \colon \F_2^n \to \F_2^m$, if $n$ is even, we let 
\[
\bentcomps{F} = \set{v \in \F_2^m : v \cdot F \text{ is bent}}.
\]
If $n$ is odd, let 
\[
\nearbentcomps{F}  =\set{v \in \F_2^m : v \cdot F \text{ is near-bent}},
\]
where a function is defined to be \textit{near-bent} if it has linearity $2^{\frac{n+1}{2}}$.

We recall the following two motivating examples seen in the previous section.
The first was \Cref{lem:indbentcpts-upperbound}, where we saw that $\deg_{alg}(1_{\bentcomps{F}}) \leq \min\set{m,\frac{n}{2}}$ when $F$ is an $(n,m)$-function that is plateaued, with $n$ being even.
The second was in our proof of \Cref{thm:parityadjoint}, where we saw that $\deg_{alg}(1_{\bentcomps{F|_H}}) = \deg_{alg}(e \cdot \pi_F^{-1}) \leq \frac{n-1}{2}$ when $n \geq 5$ is odd, $F$ is a quadratic APN function, and $H$ is an affine hyperplane with underlying vector space $\set{0,e}^\perp$.

Let us first consider plateaued APN functions of an even number of variables. 
It is well-known that if $F \colon \F_2^n \to \F_2^n$ is plateaued APN and $n \geq 4$ is even, then $|\bentcomps{F}| \equiv 2 \pmod 4$ (to the best of the authors' knowledge, this fact was first observed \cite[Proof of Theorem 2]{Mesnager2019}, and it was also later proven in \cite[Proposition 6.5]{KolschPolujan2026}).
Then, in \cite[Proof of Proposition 4.14]{MihailaThornburgh2026} it was shown that for $n\geq 4$ even and a plateaued APN function $F \colon \F_2^n \to \F_2^n$, we have $\widehat{1_{\bentcomps{F}}}(u) \equiv 2 \pmod 4$ for all $u \in \F_2^n$.
In the following lemma, we demonstrate that $|\bentcomps{F}|\equiv 2 \pmod 4$ implies $1_{\bentcomps{F}}$ has algebraic degree at least $\frac{n}{2}$ by recalling a well-known theorem of McEliece.

\begin{lemma}\label{lem:indbentcpts-lowerbound}
    Assume $n \geq 4$ is even, and let $F \colon \F_2^n \to \F_2^n$ be a plateaued APN function.
    Then $\deg_{alg}(1_{\bentcomps{F}}) \geq \frac{n}{2}$.
\end{lemma}
\begin{proof}
    McEliece's theorem \cite{McEliece1972} states every codeword of the Reed-Muller code $RM(r,n)$ (the codewords of $RM(r,n)$ are equivalently defined as the $n$-variable functions of algebraic degree at most $r$) has weight divisible by $2^{\floor{\frac{n-1}{r}}}$.
    So, for an $n$-variable Boolean function $f$ of algebraic degree $r < \frac{n}{2}$, we know that $2r \leq n-1$, and so $2^{\floor{\frac{n-1}{r}}} \geq 4$.
    Since $n \geq 4$, we know that $|\bentcomps{F}|\equiv 2 \pmod 4$, implying $1_{\bentcomps{F}}$ has algebraic degree at least $\frac{n}{2}$.
\end{proof}

\begin{remark}
    \Cref{lem:indbentcpts-upperbound} provides an upper bound on $\deg_{alg}(1_{\bentcomps{F}})$ for a plateaued function $F$, for $n$ even, but \Cref{lem:indbentcpts-upperbound} does not provide an explicit argument for this.
    However, assuming $F$ is a quadratic $(n,m)$-function and $n$ is even, we can explicitly show that $1_{\bentcomps{F}}$ is equal to a sum of products of $\frac{n}{2}$ linear Boolean functions when $n$ is even and $F$ is quadratic.
    For all $v \in \F_2^m$, let $\ell_v \colon \F_2^n \to \F_2$ be the affine function such that for all $x \in \F_2^n$, we have
    $v\cdot F(x) =\ell_v(x) \oplus \bigoplus_{1 \leq i < j \leq n}a_{i,j}^v x_ix_j$
    for some $a_{i,j}^v \in \F_2$.
    For any $v \in \F_2^m$, let $M_v$ be the symmetric $n \times n$ matrix with zero diagonal such that its $(i,j)$-entry with $i < j$ is $a_{i,j}^v$.
    It is known \cite[Section 6.1.13]{CarletBook} that $M_v$ is invertible if and only if $v \cdot F$ is bent  because the kernel of $M_v$ is the linear kernel of $v \cdot F$ (and since $v \cdot F$ is partially-bent we know it is bent if and only if its linear kernel is equal to $\set{0}$).
    That is, we have $1_{\bentcomps{F}}(v) = \det(M_v)$, for all $v \in \F_2^m$.
    By the Leibniz formula, we have $\det(M_v) = \sum_{\sigma \in S_n} \prod_{k=1}^n (M_v)_{k, \sigma(k)}$, where $S_n$ denotes the symmetric group of $\set{1, \dots, n}$.
    Note that if $\sigma \in S_n$ has a fixed point, then $\prod_{k=1}^n(M_v)_{k,\sigma(k)} =0$ since $M_v$ has zero diagonal.
    Also, if $\sigma$ has a cycle $C=(i_1, \dots, i_r)$ of length at least $3$, then the permutation $\tau \neq \sigma$ obtained by reversing $C$ satisfies 
    $\prod_{k=1}^n (M_v)_{k, \sigma(k)} = \prod_{k=1}^n (M_v)_{k, \tau(k)}$ because 
    \[
    (M_v)_{i_1, i_2}(M_v)_{i_2, i_3} \cdots (M_v)_{i_r, i_1} = (M_v)_{i_1, i_r} (M_v)_{i_r, i_{r-1}} \cdots  (M_v)_{i_2, i_1}.
    \]
    Therefore, denoting by $T$ the subset of $S_n$ consisting of all involutions that have no fixed points, we have
    $\det(M_v) = \sum_{\sigma \in T} \prod_{k=1}^n (M_v)_{k, \sigma(k)}$.
    For $\sigma = (i_1, j_1) \cdots (i_{n/2}, j_{n/2}) \in T$, we have 
    \[
     \prod_{k=1}^n (M_v)_{k, \sigma(k)}
    = \prod_{\ell=1}^{n/2} (M_v)_{i_\ell, j_\ell}(M_v)_{j_\ell, i_\ell}
    = \prod_{\ell=1}^{n/2} (M_v)_{i_\ell, j_\ell} 
    = \prod_{\ell=1}^{n/2} a^v_{\min\set{i_\ell, j_\ell}, \max\set{i_\ell, j_\ell}}.
    \]
    Note that for any $1 \leq i < j \leq n$, the $m$-variable Boolean function $v \mapsto a_{i,j}^v$ is linear.
    Since $1_{\bentcomps{F}}(v) = \det(M_v)$ for all $v \in \F_2^m$, it follows that $1_{\bentcomps{F}}$ is the sum of products of $\frac{n}{2}$ linear functions, and so $\deg_{alg}(1_{\bentcomps{F}}) \leq \frac{n}{2}$.
\end{remark}

We now completely determine the algebraic degree of the indicator of the bent components of any plateaued APN function of an even number of variables.

\begin{proposition}\label{prop:platAPN-bentcomps-algdeg}
    Assume $n \geq 4$ is even, and let $F \colon \F_2^n \to \F_2^n$ be a plateaued APN function.
    Then $\deg_{alg}(1_{\bentcomps{F}}) = \frac{n}{2}$.
\end{proposition}
\begin{proof}
    Apply the upper and lower bounds in \Cref{lem:indbentcpts-upperbound} and \Cref{lem:indbentcpts-lowerbound}, respectively.
\end{proof}

\Cref{prop:platAPN-bentcomps-algdeg} implies that for $n \geq 4$ even and a plateaued APN $(n,n)$-function $F$, there exists an $\frac{n}{2}$-dimensional affine subspace $A \subseteq \F_2^n$ such that $|\bentcomps{F} \cap A|$ is odd (note that we can also require the stronger condition that $A$ is a linear subspace), and moreover, $\bentcomps{F}$ meets every $(\frac{n}{2}+1)$-dimensional affine subspace in an even number of points.
Moreover, in the case that $F$ is quadratic, we obtained an exact description of $1_{\bentcomps{F}}$ from the coefficients of quadratic terms in the ANFs of the components of $F$.

In the following proposition, we prove a similar result but for the near-bent component functions of the restrictions of quadratic APN functions to hyperplanes when $n$ is even.
The proof of the following result is more involved than the proof of \Cref{prop:platAPN-bentcomps-algdeg} because it will require the usage of \Cref{thm:parityadjoint}, as well as the fact that the component functions of the ortho-derivative of a quadratic APN function have algebraic degree $n-2$.

\begin{proposition}\label{prop:nearbent-restriction-ind}
    Let $n \geq 4$ be even, and let $F \colon \F_2^n \to \F_2^n$ be a quadratic APN function.
    Let $u \in \F_2^n$ be nonzero, and let $F'=F|_H$ where $H = \set{0,u}^\perp$.
    Then, $\deg_{alg}(1_{\nearbentcomps{F'}})=n-1$.
\end{proposition}
\begin{proof}
For all $v \in \F_2^n$, let $\lambda_{v,H}$ denote the amplitude of $v \cdot F$ on $H$ (recall that $v \cdot F$ is plateaued on $H$ because it is partially-bent \cite{CarletThornburghRestrictions}), and let $k_{v,H}$ be an integer such that $\lambda_{v,H}^2 = 2^{n+k_{v,H}}$.
For any nonzero $v\in \F_2^n$, the function $v \cdot F|_H$ is in $n-1$ variables and $n-1$ is odd, so $\lambda_{v,H} \geq 2^{\frac{(n-1)+1}{2}}=2^{\frac{n}{2}}$ with equality if and only if $v \cdot F|_H$ is near-bent, i.e. $1_{\nearbentcomps{F'}}(v) = 1$.
Note that $k_{v,H}$ is non-negative and even for all $v \in \F_2^n$.
Letting $t \in  H$ and $v \in \F_2^n \setminus \set{0}$, we have by \Cref{rem:plateaued-restriction-mults} that
$6\mult_S(t,F(t)+v) 
= \sum_{b \in \F_2^n}(-1)^{ b\cdot v} 2^{k_{v,H}}$, where $S = \graph{F'}$, and this implies $2 \mult_S(t,F(t)+v) \equiv  \sum_{b \in \nearbentcomps{F'}}(-1)^{b\cdot v} \pmod 4$.
The latter sum is equal to $\widehat{1_{\nearbentcomps{F'}}}(v)$, so $\mult_S(t,F(t)+v)$ is odd if and only if $\widehat{1_{\nearbentcomps{F'}}}(v) \equiv 2 \pmod 4$.

By \Cref{thm:parityadjoint}, we know $F$ has an exclude parity adjoint $\varepsilon_F$, implying 
\[
\set{v \in \F_2^n \setminus \set{0} : \mult_{\graph F \cap (\set{0,u}^\perp \times \F_2^n)}(0,F(0)+v) \text{ is even}}
\]
is an affine hyperplane (note that we are using the fact $F$ is quadratic and not only crooked), and this affine hyperplane is equal to $\F_2^n \setminus \set{0,\varepsilon_F(u)}^\perp$.
Thus, $\widehat{1_{\nearbentcomps{F'}}}(v) \equiv 0 \pmod 4$ for all $v \in \F_2^n \setminus \set{0,\varepsilon_F(u)}^\perp$ and $\widehat{1_{\nearbentcomps{F'}}}(v) \equiv 2 \pmod 4$ for all nonzero $v \in \set{0,\varepsilon_F(u)}^\perp$.
Note that $1_{\nearbentcomps{F'}}$ has even Hamming weight because $|\nearbentcomps{F'}| \equiv \widehat{1_{\nearbentcomps{F'}}}(v) \pmod 2$ for all $v \in \F_2^n$, implying $\deg_{alg}(1_{\nearbentcomps{F'}}) \leq n-1$.
Note that 
\[
|\nearbentcomps{F'} \cap \set{x \in \F_2^n : x \cdot v =1}|= \frac{ |\nearbentcomps{F'}| - \widehat{1_{\nearbentcomps{F'}}}(v)}{2},
\]
and in both cases of $|\nearbentcomps{F'}| \equiv 0 \pmod 4$ and $|\nearbentcomps{F'}| \equiv 2 \pmod 4$, we see that it is possible to find a nonzero $v \in \F_2^n$ such that $|\nearbentcomps{F'} \cap \set{x \in \F_2^n : x \cdot v =1}|$ is odd because $\widehat{1_{\nearbentcomps{F'}}}$ takes values that are congruent to $0$ and $2$ modulo $4$ on $\F_2^n \setminus \set{0}$.
Hence, $\deg_{alg}(1_{\nearbentcomps{F'}}) \geq n-1$, and we conclude the proof.
\end{proof}

\begin{remark}
    It is known that when $n \in \set{6,8}$, there exists a quadratic APN $(n,n)$-function of linearity $2^{n-1}$, the largest possible linearity of a plateaued APN function, but it is unknown whether or not there exist more quadratic APN functions of maximal linearity \cite{BeierleLeanderExtensions}. 
    One interesting characteristic of a quadratic APN function $F \colon \F_2^n \to \F_2^n$ of linearity $2^{n-1}$ is that there exists an affine hyperplane $H \subseteq \F_2^n$ such that $F|_H \colon H \to \F_2^n$ is EA equivalent to $x \mapsto i(G(x))$ for some quadratic APN function $G \colon\F_2^{n-1} \to \F_2^{n-1}$ and linear embedding $i \colon \F_2^{n-1} \to \F_2^n$ \cite{BeierleLeanderExtensions}.
    This makes \Cref{prop:nearbent-restriction-ind} straightforward for this particular case of $F'=F|_H$.
    Indeed, all component functions of $G$ are near-bent since $G$ is AB, and so $\deg_{alg}(1_{\nearbentcomps{G}})=n-1$, and it is then straightforward to see that $\deg_{alg}(1_{\nearbentcomps{F'}}) = n-1$.
\end{remark}
\medskip 
\noindent{\bf Acknowledgment:} We thank Enrico Piccione for an interesting discussion which led us to Paragraph \ref{Piccione}.

\bibliographystyle{plain}
\bibliography{bibliography}
\end{document}